\documentclass{daj}

\usepackage{ amssymb, amsmath, enumerate, amsfonts, amsthm, mathrsfs, url, bm}

\numberwithin{equation}{section}

\usepackage{mystyle}

\dajAUTHORdetails{%
  title = {Quantitative Bounds in the Polynomial Szemer\'edi Theorem:\ the Homogeneous Case}, 
  author = {Sean Prendiville},
  plaintextauthor = {Sean Prendiville},
  plaintexttitle = {Quantitative Bounds in the Polynomial Szemeredi Theorem: The Homogeneous Case}, 
  runningtitle = {Quantitative Bounds in the Polynomial Szemer\'edi Theorem}, 
  keywords = {Bergelson--Leibman theorem, polynomial Szemer\'edi, Gowers norms, density bounds},
}   

\dajEDITORdetails{%
   year={2017},
   number={5},
   received={19 January 2017},   
   published={21 February 2017},  
   doi={10.19086/da.1282},       
}   

\begin{document}

\begin{frontmatter}[classification=text]


\author[sp]{Sean Prendiville\thanks{Whilst carrying out this work the author was supported by a Richard Rado Postdoctoral Fellowship at the University of Reading.}}

\begin{abstract}
We obtain quantitative bounds in the polynomial Szemer\'{e}di theorem of Bergelson and Leibman, provided the polynomials are homogeneous and of the same degree.  Such configurations include arithmetic progressions with common difference equal to a perfect $k$th power.
\end{abstract}
\end{frontmatter}
\setcounter{tocdepth}{1}
\tableofcontents

\newpage
\section{Introduction}

Szemer\'edi's theorem \cite{szemeredi75} has been called a `Rosetta stone' \cite{taoICM} connecting various areas of mathematics, a consequence of the diversity of proofs it has received.  A generalisation of this theorem proved by Bergelson and Leibman \cite{bergelsonleibman96} states that for polynomials $P_1, \dots, P_n \in \Z[x]$ with zero constant term, a set $A\subset [N]$ lacking the configurations 
\begin{equation}\label{BL config}
x, \ x+P_1(y), \ \dots, \ x + P_n(y) \quad \text{with} \quad y \in \Z \setminus \set{0} 
\end{equation}
satisfies the size bound $|A| = o_{\vP}(N)$.

In contrast with Szemer\'edi's theorem, all proofs of this polynomial generalisation proceed via ergodic methods.  
On a number of occasions, Gowers \cite{gowersicm, gowerscdm, gowers01} has 
asked for an alternative proof of the polynomial Szemer\'edi theorem, in particular a proof yielding quantitative bounds.  The purpose of this article is to provide such a proof when the polynomials are homogeneous and of the same degree.

\begin{theorem}\label{main:thm}
Let $c_1, \dots, c_n \in \Z$.  If $A \subset [N] := \set{1, 2, \dots, N}$ lacks configurations of the form 
\begin{equation}\label{main config}
x,\quad x + c_1y^k,\quad \dots,\quad x+c_ny^k \quad \text{with} \quad y \in \Z \setminus \set{0}
\end{equation}
then $A$ satisfies the size bound
\begin{equation}\label{loglog bound}
\begin{split}
|A| \ll_{\vc,k} N(\log\log N)^{-c(n,k)}.
\end{split}
\end{equation}
Here $c(n,k)$ is a positive absolute constant dependent only on the length and degree of the configuration \eqref{main config}.  
\end{theorem}

The only previously known quantitative result for a non-linear configuration of size greater than two is due to Green \cite{green02}, who determined a bound of the shape \eqref{loglog bound} for the progression
$$
x,\quad x+ y^2 + z^2,\quad x + 2y^2 + 2z^2. 
$$
We obtain comparable density bounds for arbitrarily long configurations of this type, a seemingly more general consequence of Theorem \ref{main:thm}.
\begin{corollary}\label{homogeneous:thm}
Let $P_1, \dots, P_n \in \Z[y_1, \dots, y_m]$ be homogeneous polynomials, all of degree $k$, and let $K$ denote a finite union of proper subspaces of $\R^m$.  If $A \subset [N]$ lacks configurations of the form 
\begin{equation}\label{homogeneous config}
x, x + P_1(\vy), \dots, x+P_n(\vy) \quad \text{with} \quad \vy \in \Z^m \setminus K
\end{equation}
then $A$ satisfies the size bound
\begin{align*}
|A| \ll_{\vP, K} N(\log\log N)^{-c(n,k)}.
\end{align*}
\end{corollary}

A first step towards a quantitative Bergelson--Leibman theorem was taken by Walters \cite{walters}, who obtained a combinatorial proof of the polynomial van der Waerden theorem.  This was originally established in \cite{bergelsonleibman96} and asserts the existence of the following number.   
\begin{definition}[Polynomial van der Waerden number]
Given integer polynomials $P_1, \dots, P_n \in \Z[y]$ with zero constant term, define the \emph{van der Waerden} number $W(\vP, r)$ to be the least positive integer $N$ such that any $r$-colouring of $[N]$ results in a monochromatic configuration of the form \eqref{BL config}.
\end{definition} 

Walters's argument generalises the colour-focusing argument of van der Waerden \cite{vdw}.  As a consequence, when, for instance, $\vP$ corresponds to a $k$th power progression
\begin{equation}\label{kth AP}
x,\quad x+y^k,\quad x+2y^k,\quad \dots,\quad x+ (n-1)y^k
\end{equation} 
the argument yields (at best) an Ackerman-type bound on $W(\vP, 2)$ in terms of the length $n$.  
Unlike Shelah's \cite{shelah} primitive recursive bounds in van der Waerden's theorem, Gowers \cite[p.186]{gowerscdm} has observed:
\begin{quote}
It seems not to be possible to find a `Shelah-ization' of Walters's proof, so it is still an open problem whether the bounds can be made primitive recursive.\hfill 
\end{quote}
Theorem \ref{main:thm} yields the first `reasonable' bounds on $W(\vP, r)$ in terms of $r$.
\begin{corollary}
If $P_i = c_i y^k$ for $i =1, \dots, n$ then there exist constants $C_1 = C_1(\vP)$ and $C_2 = C_2(n, k)$ such that
\begin{equation*}
\begin{split}
W(\vP, r) \leq \exp \exp\brac{ C_1{r^{C_2}}}
\end{split}
\end{equation*}
\end{corollary}
\noindent This follows from the fact that in any $r$-colouring of $[N]$ there is a colour class of size at least $N/r$.

Determining an upper bound for $W(\vP, 2)$ in terms of $n$ is a more delicate matter.  To answer this question using the methods of this paper requires one to make all constants of the form $C(n, k)$ explicit, at the risk of obfuscating the essential ideas.  It would be interesting to determine whether such an approach gives the first primitive recursive bounds for the configuration \eqref{kth AP} with $k$ fixed.
\begin{conjecture}\label{tower conjecture}
If $\vP$ corresponds to an arithmetic progression of length $n$ with square common difference
\begin{equation}\label{short square AP}
x,\quad x+y^2,\quad x+2y^2,\quad \dots,\quad x+ (n-1)y^2,
\end{equation}
then the function $W(\vP, 2)$ is bounded above by a tower of twos of height $n + 5$.  
\end{conjecture}

See \S\ref{quantitative dependence} for evidence towards this.  Much stronger bounds should hold, but it seems unlikely that the methods of this paper suffice for their deduction.  For comparison, when $\vP$ corresponds to an arithmetic progression of length $n$, Gowers \cite{gowers01} has established that
$$
W(\vP, 2) \leq 2\uparrow2\uparrow2\uparrow2\uparrow2\uparrow(n+9).
$$
Here $a \uparrow b$ denotes $a^b$.

Previous results of the type recorded in Theorem \ref{main:thm} concern either linear configurations (when $k = 1$) or two-point non-linear configurations (when $n = 1$).  For the linear case, the first bound for three-term progressions was obtained by Roth \cite{roth}, and for longer configurations by Gowers \cite{gowers01}.  It is this latter  approach  we generalise.  Gowers in fact provides an explicit estimate for the exponent appearing in \eqref{loglog bound}, namely
\begin{equation}\label{gowers estimate}
\begin{split}
c(n, 1) \geq 2^{-2^{n+9}}
\end{split}
\end{equation}
Replicating this when $k > 1$ entails the same issues encountered in addressing Conjecture \ref{tower conjecture}; see \S \ref{quantitative dependence} for more on this. 
Roth's bound has received a number of improvements, see \cite{bloom14} and the references therein, whilst Gowers's result has only been improved in the case of four-point  configurations \cite{greentao}.

For two-point non-linear configurations, the first quantitative bounds were obtained by S\'ark\H{o}zy \cite{sarkozyI, sarkozyIII} and the current records are found in a  preprint of Rice \cite{rice}, with a number of results in the interim (see the references in the latter). 

As previously remarked, for non-linear configurations of length greater than two, the only existing quantitative result  is due to Green \cite{green02}, who considers three-term progressions with difference equal to a sum of two squares.  The logarithmic density of such numbers, together with their multiplicative structure, allows for methods unavailable for the sparser configurations considered in this paper.  Employing Corollary \ref{homogeneous:thm} we obtain an alternative proof of Green's result.

The structure of our argument is discussed in detail in \S\ref{structure}.  In brief, our approach is to apply the method of van der Corput differencing to relate the non-linear configuration \eqref{main config} to a longer linear configuration 
\begin{equation}\label{linear intro config}
\begin{split}
x,\quad x+a_1y,\quad \dots,\quad  x+a_dy
\end{split}
\end{equation}
with length $d = d(n,k)$ dependent only on $n$ and $k$.  We then treat this linear configuration using the methods of Gowers \cite{gowers01}.  The use of van der Corput's inequality allows us to control the size of the coefficients $a_i$.  In essence, these deliberations establish that the polynomial progressions under consideration are controlled by an average of local Gowers norms, each localised to a subinterval.    

The main technical difficulty is that the common difference $y$ in the linear configuration \eqref{linear intro config} is constrained to lie in a much shorter interval than the shift parameter $x$.  Unfortunately, the current inverse theory for the Gowers norms can only handle parameters $x$ and $y$ ranging over similarly sized intervals.  Our strategy, heuristically at least, is to decompose $y$ into a difference of smaller parameters $y = y_1-y_0$.  Changing variables in the shift $x$, we transform the configuration \eqref{linear intro config}
into one of the form
$$
x +b_0 y_0,\quad x+c_1y_1, \quad x + b_2y_0+c_2y_1,\quad \dots, \quad x + b_dy_0+c_dy_1.
$$  
For each fixed value of $x$, one can view this as a shift of the linear configuration 
$$
b_0 y_0,\quad c_1y_1, \quad b_2y_0+c_2y_1,\quad \dots, \quad  b_dy_0+c_dy_1.
$$  
Crucially, in this linear configuration the parameters $y_0$ and $y_1$  range over the same interval.  To each of these shifted `short' configurations we apply Gowers's inverse theorem for the $U^d$-norm \cite{gowers01}, which yields a density increment on an even shorter subprogression. 

We end this introduction by showing how Corollary \ref{homogeneous:thm} follows from Theorem \ref{main:thm}.

\begin{proof}[Proof that Theorem \ref{main:thm} $\implies$ Corollary \ref{homogeneous:thm}]
Suppose that $A \subset [N]$ lacks configurations of the form \eqref{homogeneous config}.  An induction on dimension shows that $\Z^m$ is not contained in any finite union of proper (affine) subspaces of $\R^m$.  Hence there exists $\vz \in \Z^m \setminus K$.
%
%
%
Notice that we must have $y \vz \notin K$ for all $y \in \Z\setminus\set{ 0}$.  Let us define $c_i := P_i(\vz)$ for $i = 1, \dots, n$.  Then by homogeneity, the set $A$ lacks configurations of the form $x, x+c_1y^k, \dots, x+ c_n y^k$ with $y \in \Z\setminus \set{0}$.  The result now follows on employing Theorem \ref{main:thm}.
\end{proof}


\section{The structure of our argument}\label{structure}

The structure of our argument closely follows the general density increment strategy of  \cite{roth, gowers01, green02}.  Let us illustrate these ideas with respect to the configuration
\begin{equation}\label{square 3AP}
\begin{split}
x, \quad x+y^2, \quad x+2y^2 \qquad (y \in \Z\setminus \set{0}).
\end{split}
\end{equation}
Our ultimate aim is to show that  if a set $A \subset [N]$ of density $\delta := |A| / N$ lacks \eqref{square 3AP}, then there exists a long arithmetic progression with square common difference 
\begin{equation}\label{square AP}
\begin{split}
a + q^2 \cdot [N_1]
\end{split}
\end{equation}
on which $A$ has increased density.  Let $A_1$ denote the set of $x \in [N_1]$ for which $a + q^2x \in A$.  Then the fact that \eqref{square AP} has square common difference ensures that $A_1$ also lacks \eqref{square 3AP}, moreover $A_1$ has greater density on $[N_1]$ than $A$ does on $[N]$.  Iterating this argument eventually results in a configuration-free set whose density exceeds one.  This contradiction allows us to extract a quantitative bound on the density of the initial set $A$.  The proof of the density increment step occupies the majority of our paper, the more standard iteration and extraction of a final bound taking place in \S\ref{iteration}. 

Given functions $f_i : \Z \to \R$ define the trilinear operator
\begin{equation*}
\begin{split}
T(f_0, f_1, f_2) := \sum_{x\in\Z} \sum_{y \in \N} f_0(x) f_1(x+y^2) f_2(x+2y^2).
\end{split}
\end{equation*}
Then $T(1_A) := T(1_A, 1_A, 1_A)$ counts the number of configurations \eqref{square 3AP} in the set $A$, and we begin by comparing this to $T(\delta 1_{[N]})$, the expected value were $A$ a random set of density $\delta$.  A crude lower bound shows that 
\begin{equation*}
\begin{split}
T(\delta 1_{[N]}) \gg \delta^3 N^{3/2}.
\end{split}
\end{equation*}
 Hence if 
 \begin{equation}\label{little o}
\begin{split}
|T(1_A) - T(\delta 1_{[N]})|  \leq \trecip{2}T(\delta 1_{[N]})
\end{split}
\end{equation}
then $T(1_A) \gg \delta^3 N^{3/2}$.  In particular, $T(1_A) > 0$, which yields a contradiction if we are assuming that $A$ lacks \eqref{square 3AP}.

It follows that \eqref{little o} does not hold.
Write $f_A = 1_A - \delta 1_{[N]}$ for the balanced function of $A$.  Then by trilinearity there must exist 1-bounded functions $f_i :\Z \to [-1, 1]$ supported on $[N]$, at least one of which is equal to $f_A$, and such that 
\begin{equation}\label{large bal sketch}
\begin{split}
|T(f_0, f_1, f_2)| \gg \delta^3\, T( 1_{[N]}).
\end{split}
\end{equation}
For the sake of exposition, let us assume that $f_2 = f_A$.  So far these deductions are standard, and closely follow \cite{gowers01}.  

\begin{definition}[Gowers uniformity norm]
Given a function $f:\Z \to \R$ with finite support, define
\begin{equation}\label{gowers norm}
\begin{split}
\norm{f}_{U^d}^{2^d} := \sum_{h_1, \dots, h_d} \sum_x \Delta_{h_1, \dots, h_d} f(x),
\end{split}
\end{equation}
where
\begin{equation}\label{difference operator}
\begin{split}
\Delta_h f(x) := f(x+h) f(x)
\end{split}
\end{equation}
and
\begin{equation*}
\begin{split}
\Delta_{h_1, \dots, h_d} f := \Delta_{h_1} \dots \Delta_{h_d} f.
\end{split}
\end{equation*}
For $S \subset \Z$ let us define the $U^d$-norm localised to $S$ by
\begin{equation*}
\begin{split}
\norm{f}_{U^s(S)} := \norm{f1_S}_{U^s}.
\end{split}
\end{equation*}
\end{definition}
Were one able to continue as in \cite{gowers01}, one might hope to prove that there exist absolute constants $d$ and $c >0$ such that 
 \begin{equation}\label{global von neu}
\begin{split}
\sup_{|f_0|, |f_1| \leq 1_{[N]}} |T(f_0, f_1, f_2)|\ll T(1_{[N]})\brac{ \frac{\norm{f_2}_{U^d}}{\norm{1_{[N]}}_{U^d}}}^c.
\end{split}
\end{equation}
Green and Tao \cite{greentaoprimes} call such a result a \emph{generalised von Neumann theorem}.  Combining this with \eqref{large bal sketch} gives 
\begin{equation}\label{large global norm}
\begin{split}
\norm{f_A}_{U^d} \gg \delta^C \norm{1_{[N]}}_{U^d}.
\end{split}
\end{equation}

Such a conclusion does not immediately appear useful unless $d=1$. Unlike the relatively simple $U^1$-(semi)norm
$$
\norm{f}_{U^1} = \Bigabs{\sum_x f(x)},
$$
the higher order $U^d$-norms are much harder to understand.  However, the beef of \cite{gowers01} says that largeness of these norms is explained, at least on a local level, by largeness of the $U^1$-norm.  More precisely, we have the following.

\newtheorem*{gowersthm}{Gowers's inverse theorem}
\begin{gowersthm}For $d \geq 1$ there exist constants $C = C(d)$  and $c = c(d) >0$ such that the following is true.  Suppose that
$f : \Z \to [-1,1]$ satisfies 
\begin{equation*}\label{ModUdBound:eqn}
\norm{f}_{U^d[N]} \geq \delta \norm{1}_{U^d[N]}.
\end{equation*}  
Then one can partition $[N]$ into arithmetic progressions $P_i$, of average length at least $c\delta^CN^{c\delta^{C}}$, such that
\begin{equation}\label{gowers conclusion}
\begin{split}
& \sum_i \norm{f}_{U^1(P_i)} \geq c \delta^{C} \sum_i \norm{1}_{U^1(P_i)}.
\end{split}
\end{equation}
\end{gowersthm}

Employing this in conjunction with \eqref{large global norm} provides a partition of $[N]$ into progressions $P_i$ such that 
\begin{equation}\label{large U1 balance}
\begin{split}
 \sum_i \abs{\sum_{x \in P_i} f_A(x)} \gg  \delta^{C} \sum_i |P_i|.
\end{split}
\end{equation}
On noting that
\begin{equation*}
\begin{split}
 \sum_i \sum_{x \in P_i} f_A(x) = \sum_x f_A(x) = 0, 
\end{split}
\end{equation*}
we may add this to \eqref{large U1 balance} to deduce that there exists an index $i$ such that
\begin{equation}\label{increment conclusion}
\begin{split}
\sum_{x \in P_i} f_A(x) \gg  \delta^C |P_i|.
\end{split}
\end{equation}
This yields a density increment on a subprogression.

There are two flaws with this argument:  The first is that the subprogressions given by Gowers's inverse theorem may not have square common difference as in \eqref{square AP}. Rectifying this requires a purely technical modification of \cite{gowers01}, as first demonstrated for the $U^3$-norm by Green \cite{green02}.  This is explained further in \S\ref{mod gowers section}.

The second flaw, and most problematic, is that no generalised von Neumann inequality of the form \eqref{global von neu} exists in the literature.  In recent  work of Tao and Ziegler \cite{taozieglerconc},  a qualitative version of such a result is deduced which amounts to saying that if $|T(f_0, f_1, f_A)|$ is large, then some global Gowers norm $\norm{f_A}_{U^d}$ must also be large\footnote{See also the video lecture: T.~Tao, \emph{Concatenation theorems for the Gowers uniformity norms}, BIRS workshop on Combinatorics Meets Ergodic Theory, \url{http://goo.gl/UskoBQ}.}.  However, the quantitative dependence in this is at least tower-exponential \cite{taozieglerpersonal}, and is thus  insufficient for our purpose.

The key idea of this paper is to aim for less.  Instead of showing that the  counting operator $T$ is controlled by a single {global} Gowers norm, we show that $T$ is controlled by an average of local Gowers norms, each localised to a subprogression of length approximately $\sqrt{N}$.  
\begin{definition}[Localised $U^d$-norm]
Define the $U^d$-norm localised to scale $M$ by
\begin{equation}\label{localised norm}
\begin{split}
\norm{f}_{U^d \sim M} := \sum_x \norm{f}_{U^d(x + [M])} .
\end{split}
\end{equation}
\end{definition}
This is an average of the Gowers norm of $f$ over every interval of length $M$.  
A more complicated version of this localised norm appears in work of Tao and Ziegler \cite{taoziegler}, and one can think of \eqref{localised norm} as a version of their norm in which a number of extra averaging parameters have been fixed.

Using this norm we are able to prove the following local von Neumann theorem.

\newtheorem*{localvnt}{Local von Neumann theorem}
\begin{localvnt}
Let $f_0, f_1$ be 1-bounded functions supported on $[N]$.  Suppose that
$$
|T(f_0, f_1, f_A)| \geq \delta T(1_{[N]}).
$$ 
Then, provided that $N \geq C \delta^{-C}$, there exists $M$ in the range
$$
\delta^C\sqrt{N} \ll M \ll \delta^{-C} \sqrt{N}
$$ such that we have the local non-uniformity estimate
\begin{equation}\label{averaged lower bound}
\begin{split}
\norm{f_A}_{U^7\sim M} \gg \delta^C \norm{1_{[N]}}_{U^7\sim M}.
\end{split}
\end{equation}
\end{localvnt}

The non-uniformity estimate \eqref{averaged lower bound} can be interpreted as saying that, for at least $c \delta^CN$ of the intervals $x + [M]$, we have 
\begin{equation}\label{short non-uniformity}
\begin{split}
\norm{f_A}_{U^7(x +  [M])} \gg \delta^C \norm{1}_{U^7(x +   [M])}.
\end{split}
\end{equation}
To each of these intervals, we apply Gowers's inverse theorem (suitably modified) to deduce the existence of a partition of $x + [M]$ into fairly long progressions $P_{x, i}$, each with square common difference, and such that 
\begin{equation*}
\begin{split}
 \sum_i \norm{f_A}_{U^1(P_{x, i})} \gg  \delta^{C} \sum_i \norm{1_{[N]}}_{U^1(P_{x,i})}.
\end{split}
\end{equation*}
Taking the trivial partition for the remaining intervals, it follows that for all $x$  there exists a partition of $x + [M]$ into progressions $P_{x, i}$ with square common difference such that 
\begin{equation}\label{U1 estimate}
\begin{split}
\sum_x \sum_i \Biggabs{\sum_{y \in P_{x, i}} f_A(y)}\gg  \delta^{C} \sum_x \sum_i |P_{x, i}|.
\end{split}
\end{equation}
Crucially, since $f_A$ has mean zero, we have
\begin{equation*}
\begin{split}
\sum_x \sum_i \sum_{y \in P_{x, i}} f_A(y) & = \sum_{z \in  [M]}\sum_x  f_A(x+z)\\
& =  0.
\end{split}
\end{equation*}
We may therefore add this quantity to \eqref{U1 estimate} to conclude that there exists a long arithmetic progression $P_{x,i}$ with square common difference such that
\begin{equation*}
\begin{split}
\sum_{y \in P_{x, i}} f_A(x) \gg \delta^C |P_{x, i}|.
\end{split}
\end{equation*}
This yields the required density increment.

In the remainder of this section, we outline the ideas behind the local von Neumann theorem.  The inspiration for our approach is an argument of Green--Tao--Ziegler \cite{taosarkozy} which establishes a local von Neumann theorem for the two-point configuration $x, x+y^2$, showing that it is controlled by an average of local $U^1$-norms.  We begin by sketching their argument.

Let $f_0$ be a 1-bounded function supported on $[N]$.  Then we are interested in bounding the quantity
$$
\sum_{x \in \Z} \sum_{y \in \N} f_0(x) f_A(x+y^2).
$$
Write $I$ for the interval $[\sqrt{N}]$.  By an application of the Cauchy--Schwarz inequality and a change of variables we have
\begin{equation*}
\begin{split}
\abs{\sum_{x \in \Z} \sum_{y \in \N} f_0(x) f_A(x+y^2)}^2 &\leq N \sum_x \sum_{y_1, y_2 \in I} f_A(x+y_1^2) f_A(x + y_2^2)\\
& = N \sum_{|h| < \sqrt{N}}\sum_x f_A(x-h^2) \sum_{y \in I \cap (I - h)}  f_A(x+2hy )\\
&\ll N^{3/2} \max_{|h| < \sqrt{N}}\sum_x \norm{f_A}_{U^1(x + P_h)},
\end{split}
\end{equation*}
where $P_h$ is the progression $\set{2h y : y \in I\cap(I-h)}$.  Here we have made use of the simple identity $(y+ h)^2 - y^2= 2hy + h^2$.   

As stated, there are two deficiencies with this local von Neumann inequality: the common difference of the progression $P_h$ may be zero or a non-square.  Both of these difficulties can be surmounted by replacing the Cauchy--Schwarz inequality with van der Corput's inequality (see \S\ref{vdc section} for a statement of this inequality).  In doing so, one can deduce that for any $H \leq \sqrt{N}$ we have
\begin{equation*}
\begin{split}
\abs{\sum_{x \in \Z} \sum_{y \in \N} f_0(x) f_A(x+y^2)}^2  \ll \frac{N^{3}}{H} + N^{3/2}\max_{0 < |h| < H} \norm{f_A}_{U^1(x + P_h)}.
\end{split}
\end{equation*}
This ensures that the common difference of $P_h$ is non-zero, yet it still may be a non-square.  However, since we can control the size of this common difference (it is bounded above by $2H$), we can partition $P_h$ into at most $2H$ further subprogressions of square common difference, namely $(2h)^2$.  It follows that there exists a progression $P$ with square common difference such that
\begin{equation}\label{baby vdc}
\begin{split}
\abs{\sum_{x \in \Z} \sum_{y \in \N} f_0(x) f_A(x+y^2)}^2  \ll \frac{N^{3}}{H} + N^{3/2}H \sum_x\norm{f_A}_{U^1(x + P)}.
\end{split}
\end{equation}

Taking $H = C\delta^{-C}$ in \eqref{baby vdc}, the assumption
 \begin{equation*}
\begin{split}
\abs{\sum_{x \in \Z} \sum_{y \in \N} f_0(x) f_A(x+y^2)}  \geq \delta N^{3/2}
\end{split}
\end{equation*}
implies that
\begin{equation}\label{two point increment conclusion}
\begin{split}
\sum_x \norm{f_A}_{U^1(x+ P)} \gg \delta^C N^{3/2}.
\end{split}
\end{equation}
Employing the trivial estimate $|f_A| \leq 1_{[N]}$ and assuming that $N \geq C \delta^{-C}$, the left-hand side of \eqref{two point increment conclusion} is at most $O(N|P|)$, which gives the lower bound
$$
|P| \gg \delta^C N^{1/2}.
$$
  The corresponding upper bound $|P| \leq \sqrt{N}$ follows since $P \subset \{2hy : y \in [\sqrt{N}]\}$. We have therefore deduced the following.
  
  \newtheorem*{2localvnt}{Two-point local von Neumann}
\begin{2localvnt}
Let $f_0$ be a 1-bounded function supported on $[N]$.  Suppose that
$$
\abs{\sum_{x \in \Z} \sum_{y \in \N} f_0(x) f_A(x+y^2)} \geq \delta N^{3/2}.
$$ 
Then, provided that $N \geq C \delta^{-C}$, there exists a progression $P$ with square common difference and length 
$$
\delta^C\sqrt{N} \ll |P| \leq \sqrt{N}
$$ such that we have the local non-uniformity estimate
\begin{equation*}
\begin{split}
\sum_x \norm{f_A}_{U^1(x+ P)} \gg \delta^C \sum_x \norm{1_{[N]}}_{U^1(x+ P)}.
\end{split}
\end{equation*}
\end{2localvnt}  
 
 For longer configurations, one must employ the van der Corput inequality considerably more times.  Let us illustrate this for the inhomogeneous counting operator
 \begin{equation}
\begin{split}
\tilde{T}(f_0, f_1, f_2) := \sum_{x \in \Z} \sum_{y \in \N} f_0(x) f_1(x+y) f_2(x+y^2).
\end{split}
\end{equation}
A single application of van der Corput's inequality gives some $0 < |h_1| < H$ and some interval $I \subset [\sqrt{N}]$ for which 
$$
 |\tilde{T}(f_0, f_1, f_2)|^2  \ll \frac{N^{3}}{H} + 
 N^{3/2}\abs{\sum_x \sum_{y \in I}  f_1(x) f_1(x+h_1) f_2(x+y^2-y)f_2(x+(y+h_1)^2-y )}.
$$

To save on notation, let us write $\tilde{f}$ for a function of the form $x \mapsto f(x+b)$ for some fixed integer $b$.  Different occurrences of $\tilde{f}$ in the same equation may refer to different values of $b$, but no confusion should arise. 
A second application of van der Corput's inequality gives the existence of some 
  $0 < |h_2| < H$ and a second interval $I' \subset I$ such that 
$$
 |\tilde{T}(f_0, f_1, f_2)|^4  \ll \frac{N^{6}}{H} +  
N^{9/2} \abs{\sum_x \sum_{y \in I'}   f_2(x)\tilde{f}_2(x+2h_1y )\tilde{f}_2(x+2h_2y)\tilde{f}_2(x+2(h_1+h_2)y)}.
$$

Suppose that 
\begin{equation}\label{tilde T estimate}
\begin{split}
|\tilde{T}(f_0, f_1, f_A)| \geq \delta N^{3/2}.
\end{split}
\end{equation}
 Then as before, taking $H := C\delta^{-C}$, we can conclude the existence of non-zero integers $|a_i| \ll \delta^{-C}$ and $M \leq \sqrt{N}$ such that
\begin{equation}\label{linear estimate}
\begin{split}
 \abs{\sum_x \sum_{y \in [M]}  f_A(x)\tilde{f}_A(x+a_1y )\tilde{f}_A(x+a_2y)\tilde{f}_A(x+a_3y)} \gg \delta^CN^{3/2}.
\end{split}
\end{equation}
The trivial estimate $|f_A| \leq 1_{[N]}$ also yields the lower bound $M \gg \delta^{C} \sqrt{N}$.  Our next step is to convert \eqref{linear estimate} into a local non-uniformity estimate.  Let us demonstrate how this is done for the simpler linear average
\begin{equation}\label{3pt linear average}
\begin{split}
\abs{\sum_x \sum_{y \in [M]}   f_A(x)f_A(x+y )f_A(x-y)} \gg \delta^C NM.
\end{split}
\end{equation}

Let $M_1 \leq M$.  Then one can re-write the inner sum in \eqref{3pt linear average} as within $O(M_1)$ of
\begin{multline*}
M_1^{-2} \sum_{y_1, y_2 \in [M_1]}\ \sum_{y \in [M]-y_1+ y_2 } f_A(x)f_A(x+y )f_A(x-y )= \\  M_1^{-2} \sum_{y_1, y_2 \in [M_1]} \sum_{y \in [M] } f_A(x)f_A(x+y -y_1 +y_2)f_A(x-y +y_1 -y_2) .
\end{multline*}
Changing variables in $x$ and maximising over $y$, we deduce that the left-hand side of \eqref{3pt linear average} is at most 
\begin{equation*}
\begin{split}
\frac{M}{M_1^{2}}\abs{ \sum_x \sum_{y_1, y_2 \in [M_1]}  {f_A}(x+y_1)\tilde{f}_A(x+y_2)\tilde{f}_A(x+2 y_1 -y_2)} + O(NM_1).
\end{split}
\end{equation*}
Taking $M_1 = c \delta^CM$, inequality \eqref{3pt linear average} implies that
\begin{equation}\label{balanced 3pt linear}
\begin{split}
\abs{ \sum_x \sum_{y_1, y_2 \in [M_1]}  {f}_A(x+y_1)\tilde{f}_A(x+y_2)\tilde{f}_A(x+2 y_1 -y_2)} \gg \delta^C NM_1^2
\end{split}
\end{equation}
For fixed $x$ define the functions $g_1(y) := {f}_A(x + y)1_{[M_1]}(y)$, $g_2(y) := \tilde{f}_A(x + y)1_{[M_1]}(y)$ and $h(y) := \tilde{f}_A(x+ y) 1_{[-2M_1, 2M_1]}(y)$.  Then by orthogonality
\begin{equation*}
\begin{split}
\sum_{y_1, y_2 \in [M_1]} g_1(y_1) g_2(y_2) h(2y_1 -y_2) & = \int_{\T} \hat{g}_1(-2\alpha)  \hat{g}_2(\alpha)  \hat{h}(\alpha)\intd \alpha,
\end{split}
\end{equation*}
where we have defined the Fourier transform by
\begin{equation}\label{fourier transform}
\begin{split}
\hat{g}(\alpha) := \sum_x g(x) e(\alpha x).
\end{split}
\end{equation}
Using H\"older's inequality, Parseval and the (easily checked) identity $\norm{g}_{U^2(\Z)}$  $= \norm{\hat{g}}_{L^4(\T)}$, we deduce that
\begin{equation*}
\begin{split}
\sum_{y_1, y_2 \in [M_1]} g_1(y_1) g_2(y_2) h(2y_1 -y_2) & \leq \norm{g_1}_{L^2(\Z)}\norm{g_2}_{U^2(\Z)}  \norm{h}_{U^2(\Z)}.
\end{split}
\end{equation*}
Since $\norm{g_1}_{L^2(\Z)} \leq M_1^{1/2}$ and $\norm{g_2}_{U^2(\Z)} \leq M_1^{3/4}$ and $\delta^C \sqrt{N}\ll M_1 \leq \sqrt{N}$ we can set $I = [-2M_1, 2M_1]$ and conclude that 
\begin{equation*}
\begin{split}
\sum_x \norm{f_A}_{U^2(x + I)} \gg \delta^C N^{11/8}
\end{split}
\end{equation*}
for some interval $I$ satisfying $\delta^C\sqrt{N} \ll |I| \ll \delta^{-C} \sqrt{N}$.  This gives us our local von Neumann estimate.
 
 A similar argument can be made to work for the longer linear average \eqref{linear estimate}, replacing the $U^2$-norm with the $U^3$-norm.  The general argument for arbitrarily long linear configurations is carried out in \S\ref{local von neu section}.
 \subsection{Quantitative dependence on $n$ and $k$}\label{quantitative dependence}
The linearisation process, which takes a large non-linear polynomial average such as \eqref{tilde T estimate}, and converts it into a large linear average \eqref{linear estimate}, is generalisable and carried out in detail in \S\ref{linearisation section}.  As the complexity of the polynomial configuration increases, the number of applications of van der Corput's inequality increases inordinately.  For instance, for arithmetic progressions with square common difference \eqref{short square AP} we have the following table.

\begin{center}
\begin{tabular}{|p{2cm}|p{5cm}|p{5cm}|}
\hline
Progression length & Number of differencing steps required & Resulting degree of local Gowers norm $U^d$\\ 
\hline  
2 & 1 & 1  \\
\hline 
3 &  3 &  7\\
\hline
4 & 11 & 2047\\ \hline
$n$ &  $A_n := A_{n-1} +2^{A_{n-1}}$ with $A_1 :=0$ & $2^{A_n} - 1< \underbrace{2\uparrow  \dots \uparrow 2}_{\text{height $n$}}$\\ \hline
\end{tabular}
\end{center}
 Writing $T(n)$ for a tower of twos of height $n$, the bound $A_n + 1 \leq T(n-1)$ follows inductively via
 $$
 A_n+1 = A_{n-1} + 1 + 2^{A_{n-1}} \leq T(n-2) + \trecip{2} T(n-1) \leq T(n-1).
 $$

 Ensuring that the linearisation process does in fact terminate requires a fairly abstract inductive scheme to be carried out, a scheme which is essentially the PET-induction of Bergelson--Leibman \cite{bergelsonleibman96}.  The number of steps in this process grows so rapidly that, for general configurations, we have refrained from estimating the explicit quantitative dependence on $n$ and $k$  in all absolute constants appearing in the deduction of \eqref{linear estimate} from \eqref{tilde T estimate}.  This is one reason for the absence of a lower bound for $c(n, k)$ in Theorem \ref{main:thm}, in contrast to Gowers's estimate \eqref{gowers estimate}.  A second obstacle is that one must make explicit all constants appearing in our modification of Gowers's inverse theorem (Theorem \ref{ModInverse:thm}).

\subsection{Homogeneous versus inhomogeneous}
Although we have sketched how to prove a local von Neumann theorem for the inhomogeneous configuration 
\begin{equation}\label{inhom}
x,\quad x+y,\quad x+y^2,
\end{equation} 
this configuration is not covered by Theorem \ref{main:thm}.  More generally, we demonstrate in \S\S\ref{vdc section}--\ref{local von neu section} that a local von Neumann theorem can be proved for \emph{any} configuration of the form $x, x+P_1(y), \dots, x+P_n(y)$, where $P_i \in \Z[y]$.  The main obstacle to obtaining density bounds for sets lacking inhomogeneous configurations such as \eqref{inhom} is the density increment step.  To see this, note that if $A\subset [N]$ lacks \eqref{inhom} and has a density increment on progression of the form $x + q \cdot [N_1]$, then defining $A_1$ to be the set of $y \in [N_1]$ such that $x+qy \in A$, we see that $A_1$ lacks configurations of the form 
\begin{equation}\label{trivial config}
\begin{split}
x, \quad x+y, \quad x+qy^2 \qquad (y \in \Z \setminus \set{0}).
\end{split}
\end{equation}
The problem here is that the methods we have discussed (primarily Gowers's inverse theorem) deliver an increment on a progression with common difference $q$ which is likely to be \emph{much} larger than its length $N_1$.  As a result, \emph{every} subset of $[N_1]$ lacks the configuration \eqref{trivial config}, and there is no possibility of iterating the density increment argument.

 The remaining sections are occupied with proving a rigorous version of the sketch outlined in this section.
 
\section{van der Corput differencing}\label{vdc section}

The aim of this section and its sequel is to show how a large non-linear average 
\begin{equation}\label{non-linear average}
\begin{split}
\sum_{x \in \Z} \sum_{y \in \N} f_0(x) f_1(x +c_1y^k) \dotsm f_n(x+c_n y^k)
\end{split}
\end{equation}
 leads to a large \emph{linear} average, albeit over a longer configuration.  This deduction proceeds via van der Corput differencing, each application of which bounds a polynomial average such as \eqref{non-linear average} by a polynomial average of smaller degree.  The precise notion of degree is introduced in \S\ref{linearisation section}; in this section we confine ourselves to describing the  differencing step.

 \begin{lemma}[van der Corput inequality]\label{lem:vdc}
Let $g : \Z \to \C$ be a function supported on a finite set $\mathcal{S}\subset \Z$.  Given a finite set $\mathcal{H} \subset \Z$, write $r_\mathcal{H}(h)$ for the number of pairs $(h_1, h_2) \in \mathcal{H}^2$ such that $h_1-h_2 = h$.  Then we have the estimate
\begin{equation*}
\Bigabs{\sum_y g(y)}^2 \leq \frac{|\mathcal{S} - \mathcal{H}|}{|\mathcal{H}|^2} \sum_{h }r_\mathcal{H}(h)\sum_y g(y+h)\overline{g(y)}.
\end{equation*}
\end{lemma}

\begin{proof}
By a change of variables, for any $ h \in \Z$ we have 

$$
\sum_y g(y) =  \sum_y g(y+h).
$$ 
Averaging over $h \in \mathcal{H}$ and interchanging the order of summation gives
$$
\sum_y g(y) = \recip{|\mathcal{H}|}  \sum_y \sum_{ h \in \mathcal{H}}g(y+h).
$$
The function 
$$
y \mapsto \sum_{h \in \mathcal{H}} g(y+h)
$$
is supported on the difference set $\mathcal{S} - \mathcal{H}$.  Squaring and applying Cauchy--Schwarz, we deduce that
\begin{align*}
\Bigabs{\sum_y g(y)}^2 & \leq \frac{|\mathcal{S} - \mathcal{H}|}{|\mathcal{H}|^2}  \sum_y \sum_{ h_1, h_2 \in \mathcal{H}}g(y+h_1)\overline{g(y+h_2)}\\
& = \frac{|\mathcal{S} - \mathcal{H}|}{|\mathcal{H}|^2}  \sum_{ h_1, h_2 \in \mathcal{H}}\sum_yg(y+h_1-h_2)\overline{g(y)}\\
& = \frac{|\mathcal{S} - \mathcal{H}|}{|\mathcal{H}|^2}  \sum_{h }r_\mathcal{H}(h)\sum_yg(y+h)\overline{g(y)}.
\end{align*}
\end{proof}


%
%

\begin{lemma}[weak van der Corput]\label{weak vdc}  Suppose that $g : \Z^2 \to [-1, 1]$ is supported on $[N]\times [M]$ with $N, M \geq 1$.  Let $1 \leq H\leq M$ and let $\mathcal{H}_1\subset \Z$ denote a set containing $0$.  Then there exists $h \in [H]\setminus \mathcal{H}_1$ such that
\begin{equation*}
\sum_x\Bigbrac{\sum_y g(x,y)}^2 \ll \frac{NM^2|\mathcal{H}_1|}{H}+ M \sum_{x,y} g(x,y+h)g(x,y).
\end{equation*}
\end{lemma}

\begin{proof}
Let us apply Lemma \ref{lem:vdc} with $g_x(y) := g(x,y)$,  $\mathcal{S} := [M]$ and $\mathcal{H} := [H]$, giving
$$
\sum_x\Bigbrac{\sum_y g(x,y)}^2 \leq \frac{2M}{\floor{H}} \sum_{h} \frac{r_{[H]}(h)}{\floor{H}}\sum_{x,y} g(x, y+h)g(x,y).
$$
A change of variables yields the identity
$$
\sum_{x,y} g(x, y+h) g(x,y) = \sum_{x,y} g(x,y)g(x,y - h).
$$
Combining this with the fact that $r_{[H]}(0) = \floor{H}$, $r_{[H]}(-h) = r_{[H]}(h)$ and $r_{[H]}(h) = 0$ if $|h| \geq H$, we have
\begin{equation}\label{weak vdc inequality}
 \sum_{h} \frac{r_{[H]}(h)}{\floor{H}}\sum_{x,y} g(x, y+h)g(x,y)  =
\sum_{x,y} g(x,y)^2 + \sum_{h \in [H]} \frac{2r_{[H]}(h) }{\floor{H}}\sum_{x,y} g(x,y+ h)g(x,y).
\end{equation}
Using the trivial estimates $|\supp(g)|\leq NM$, $r_{[H]}( h) \leq \floor{H}$ and $|\mathcal{H}_1 \cap [H]| \leq |\mathcal{H}_1| -1$, the right-hand side of \eqref{weak vdc inequality} is at most
$$
NM + 2(|\mathcal{H}_1|-1) NM +  \sum_{h \in [H]\setminus\mathcal{H}_1} \frac{2r_{[H]}(h) }{\floor{H}}\sum_{x,y} g(x,y+h)g(x,y)
$$
By the pigeon-hole principle there exists $h' \in [H]\setminus\mathcal{H}_1$ such that
$$
\sum_{h \in [H]\setminus \mathcal{H}_1} \frac{r_{[H]}(h)}{\floor{H}}\sum_{x,y} g(x,y+h)g(x,y) \leq \floor{H} \sum_{x,y} g(x,y+h')g(x,y).
$$
The required inequality follows.\end{proof}

One can think of the set $\mathcal{H}_1$ as those `bad' differencing parameters $h$  we wish to avoid.

\begin{lemma}[Linearisation step]\label{linearisation step}
Let $f_0, f_1, \dots, f_n : \Z \to [-1, 1]$ be 1-bounded functions supported on $[N]$, let $I$ be an interval of at most $M$ integers, let $\mathcal{H}_1$ be a set containing $0$ and let $P_1, \dots, P_n : \Z \to \Z$.   Then for any $H \leq M$ there exists $h \in [H] \setminus \mathcal{H}_1$ such that
\begin{multline*}
 \abs{ \frac{1}{N M}  \sum_x \sum_{y \in I} f_0(x) f_1\bigbrac{x+P_1(y)} \dotsm f_n\bigbrac{x+P_n(y)}} \ll \\
 \brac{\frac{|\mathcal{H}_1|}{H}}^{1/2} + \brac{\frac{1}{N M} \sum_x \sum_{y \in I\cap (I-h)} \prod_{\substack{1 \leq i \leq n\\ \omega \in \set{0,1}}}f_i\bigbrac{x+P_i(y + \omega h) - P_1(y)}}^{1/2}.
\end{multline*}
\end{lemma}

\begin{proof}
Shifting the argument of the functions $P_i$ if necessary, we may assume that $I \subset [M]$. Set $g(x,y) :=  f_1\bigbrac{x+P_1(y)} \dotsm f_n\bigbrac{x+P_n(y)} 1_{[N]}(x)1_{I}(y)$ and let $H \in [1, M]$.
Then by the Cauchy--Schwarz inequality and Lemma \ref{weak vdc} there exists $h \in [H]\setminus\mathcal{H}_1$ such that 
\begin{align*}
\Bigbrac{\sum_x f_0(x)\sum_{y} g(x,y)}^2
&\leq \Bigbrac{\sum_x f_0(x)^2 }\sum_x\Bigbrac{\sum_y g(x,y)}^2\\
& \ll N^2M^2|\mathcal{H}_1|H^{-1} + NM \sum_{x,y} g(x,y+ h)g(x,y).
\end{align*}
\end{proof}

\section{The linearisation process}\label{linearisation section}

In this section we iteratively apply Lemma \ref{linearisation step}, beginning with the configuration $x, x+ c_1y^k, \dots, x + c_ny^k$ and eventually obtaining a configuration of the form $x, x+a_1y, \dots, x + a_dy$.  The complexity of the intermediate configurations requires us to take an abstract approach.  Moreover, each application of the linearisation step necessitates a number of technical assumptions whose sole purpose is to guarantee that the coefficients $a_i$ in our final linear configuration are non-zero and distinct.  
Before proceeding to describe the argument in general, we illustrate the underlying ideas for the configuration $x, x+y^2, x+2y^2$.

\begin{lemma}[Linearisation for square 3APs]\label{linearisation for sqAPs}
Let $f_0, f_1, f_2 : \Z \to [-1, 1]$ be supported on $[N]$ and let $1 \leq H \leq \sqrt{N}$. Then there exists an interval $I \subset [\sqrt{N}]$ and integers $a_i,b_i$ with the $a_i$ distinct, $1 \leq a_i \ll H$ and such that
\begin{multline}\label{sq3AP 8-term conclusion}
 \Bigabs{ N^{-3/2}\sum_{x\in \Z} \sum_{y \in \N} f_0(x) f_1(x+y^2) f_2(x+2y^2)}\ll \\ H^{-1/8} + 
\Bigabs{N^{-3/2}\sum_x \sum_{y \in I} f_2(x)f_2(x+a_1y+b_1)\dotsm f_2(x+a_7 y +b_7) }^{1/8}.
\end{multline}
\end{lemma}

\begin{notation}
To avoid lengthy expressions, we write $\tilde{f}$ for a function of the form $x \mapsto f(x+b)$ for some integer $b$.  Different occurrences of $\tilde{f}$ in the same equation may refer to different values of $b$, but no confusion should arise.   
\end{notation}

\begin{proof}
Our assumption on the support of $f_i$ ensures that $f_0(x) f_1(x+y^2) f_2(x+2y^2) \neq 0$ only when $y \in[ \sqrt{N}]$.  Write $I$ for this interval, and $M$ for the number of integers it contains. By Lemma \ref{linearisation step} with $\mathcal{H}_1 = \set{0}$ there exists an integer $1 \leq h_1 \leq H$ and an interval $I_1 \subset I$ satisfying
\begin{multline}\label{3AP first differencing}
 \abs{ \frac{1}{N M}  \sum_x \sum_{y \in I} f_0(x) f_1(x+y^2) f_2(x+2y^2)} \ll 
 \frac{1}{H^{1/2}} + \\ \brac{\frac{1}{N M} \sum_x \sum_{y \in I_1} f_1(x) \tilde{f}_1(x+2h_1y_1) \tilde{f}_2(x+y^2) \tilde{f}_2\bigbrac{x+(y+h_1)^2 + 2h_1y}}^{1/2}.
\end{multline}

Re-applying Lemma \ref{linearisation step} with $\mathcal{H}_2 := \set{0}$, we may conclude that there exists an integer $1 \leq h_2 \leq H$ and an interval $I_2 \subset I_1 $ satisfying
$$
 \abs{ \frac{1}{N M}  \sum_x \sum_{y \in I} f_0(x) f_1(x+y^2) f_2(x+2y^2)} \ll 
 \frac{1}{H^{1/2}} +  \frac{1}{H^{1/4}} + \\ \brac{\frac{1}{N M} \sum_x \sum_{y \in I_2} F_1(x,y)}^{1/4}.
$$
where $F_1(x,y)$ is equal to
$$
f_1(x) \tilde{f}_1(x) \tilde{f}_2(x+y^2-2h_1y)\tilde{f}_2\bigbrac{x+(y+h_2)^2-2h_1y}  \tilde{f}_2\bigbrac{x+(y+h_1)^2 }\tilde{f}_2\bigbrac{x+(y+h_1+h_2)^2 }.
$$

A function of the form $x \mapsto f_1(x) \tilde{f}_1(x)$ is 1-bounded, supported on $[N]$ and independent of $y$.  We may therefore remove this function by re-applying Lemma \ref{linearisation step} with  $\mathcal{H}_3 := \set{0}$  to conclude that  there exists an integer $1 \leq h_3 \leq H$ and an interval $I_3 \subset I_2 $ satisfying
\begin{equation}\label{final vdced config}
 \abs{ \frac{1}{N M}  \sum_x \sum_{y \in I} f_0(x) f_1(x+y^2) f_2(x+2y^2)} \ll 
 \frac{1}{H^{1/2}} +  \frac{1}{H^{1/4}}+\frac{1}{H^{1/8}} +  \brac{\frac{1}{N M} \sum_x \sum_{y \in I_3} F_2(x,y)}^{1/8}.
\end{equation}
where $F_2(x,y)$ is equal to
\begin{multline*}
 f_2(x) \tilde{f}_2(x+2h_3y)\tilde{f}_2(x+2h_2y)\tilde{f}_2\bigbrac{x+2(h_2+h_3)y}\tilde{f}_2\bigbrac{x+4h_1y} \times \\ \tilde{f}_2 \bigbrac{x+2(2h_1 + h_3)y }\tilde{f}_2\bigbrac{x+2(h_1 + h_2)y }\tilde{f}_2\bigbrac{x+2(2h_1 + h_2 + h_3)y }.
\end{multline*}

The lemma is complete, provided the following coefficients are all distinct
\begin{equation*}
\begin{split}
2h_1,\quad h_2,\quad h_3,\quad 2h_1+h_2,\quad 2h_1 + h_3,\quad h_2 + h_3,\quad 2h_1 + h_2 + h_3.
\end{split}
\end{equation*}
Unfortunately, we cannot guarantee this with the proof as written.  However, distinctness would follow if instead of taking $\mathcal{H}_2 =  \mathcal{H}_3 = \set{0}$ we took
\begin{align*}
\mathcal{H}_2  := \set{0, 2h_1}, \qquad
\mathcal{H}_3  := \set{0, 2h_1, h_2, 2h_1 + h_2, h_2 - 2h_1, 2h_1- h_2}.
\end{align*}
Lemma \ref{linearisation step} permits this, increasing the absolute constant in \eqref{final vdced config} by a factor of at most $6^{1/8}$.
\end{proof}

 The above argument required three applications of Lemma \ref{linearisation step} in order to linearise the simplest example of a non-linear $k$th power configuration of length greater than two.  In general we require many more applications of the linearisation step, and at each stage of the iteration, it is not immediately obvious that we have reduced the `degree' of the configuration at all.  To see that we have indeed reduced an invariant associated to the configuration, we require the following definition.

\begin{definition}[Degree sequence]
Given polynomials $P_1, \dots, P_n \in \Z[x]$, let $L(P_i)$ denote the leading coefficient of $P_i$ and define 
$$
D_r(P_1, \dots, P_n) := \hash \set{ L(P_i) : \deg P_i = r}.
$$
In words, $D_r(\vP)$ is the number of of distinct leading coefficients occurring amongst the degree $r$ polynomials in $\vP$.  Let us define the \emph{degree sequence} of $\vP = (P_1, \dots, P_n)$ by
\begin{equation*}\label{degree sequence}
D(\vP) := (D_1(\vP), D_2(\vP), D_3(\vP), \dots ).
\end{equation*}
\end{definition}

\begin{definition}[Colex order]
We order degree sequences according to the colexicographical ordering, so that $D(\vP) \prec D(\vQ)$ if there exists $r \in \N$ such that $D_r(\vP) < D_r(\vQ)$ and for all $s > r$ we have $D_s(\vP) = D_s(\vQ)$.  
\end{definition}

\begin{lemma} Let $\mathcal{S}$ denote the set of sequences $(m_i)_{i \in \N}$ of non-negative integers with all but finitely many entries equal to zero.  Then colex induces a well-ordering on $\mathcal{S}$.  In particular, if $P(\vm)$ is a proposition defined on $\mathcal{S}$ satisfying
$$
\sqbrac{\brac{\forall \vm' \prec \vm} P(\vm')} \implies P(\vm),
$$
then $P(\vm)$ is true for all $\vm\in \mathcal{S}$.
\end{lemma}

\begin{proof}  We leave the reader to check that $\preceq$ is transitive, anti-symmetric and total.  We show that every non-empty subset of $\mathcal{S}$ has a least element.

Let $\mathcal{F}$ be a non-empty subset of $\mathcal{S}$.  We construct a sequence $(m_l^*)_{l \in \N}$ such that for each $k \in \N$ the set 
$$
\mathcal{F}_k := \set{ \vm \in \mathcal{F} : m_l = m_l^* \text{ for all } l \geq k}
$$
is non-empty and for any $\vm \in \mathcal{F}$ we have 
\begin{equation}\label{prec}
\begin{split}
(m_l^*)_{l \geq k} \preceq (m_l)_{l \geq k}.
\end{split}
\end{equation}
It follows that $ \vm^*$ is a least element of $\mathcal{F}$. 

Since $\mathcal{F}$ is non-empty, there exists $\vm \in \mathcal{F}$.  Write $k_0$ for the minimum index satisfying $m_l = 0$ for all $l \geq k_0$.  Then we take $m_l^*:= 0$ for all $l \geq k_0$. 

Suppose we have constructed $(m_l^*)_{l \geq k}$ with the required properties and $k > 1$.  Writing $\pi_{k-1}$ for the projection onto the $(k-1)$ coordinate, $\pi_{k-1}(\mathcal{F}_{k})$ is a non-empty set of non-negative integers, hence contains a least element $m_{k-1}^*$.     

Letting $\vm \in \mathcal{F}$, we wish to check that
\begin{equation}\label{prec2}
\begin{split}
(m_l^*)_{l \geq k-1} \preceq (m_l)_{l \geq k-1}.
\end{split}
\end{equation}
Since \eqref{prec} holds, we are done if the inequality in \eqref{prec} is strict.  We may therefore assume that
$$
(m_l^*)_{l \geq k} = (m_l)_{l \geq k}.
$$
The inequality \eqref{prec2} now follows since $m_{k-1} \in \pi_{k-1}(\mathcal{F}_k) $, and therefore $m_{k-1}^* \leq m_{k-1}$.
\end{proof}

For the next two lemmas we assume that $f_0, f_1, \dots, f_n : \Z \to [-1, 1]$ are 1-bounded functions supported on $[N]$, that $I$ is an interval of at most $M$ integers, and that $P_1, \dots, P_n \in \Z[x]$ are polynomials of height at most $H$, maximal degree $k $ and such that $0, P_1, \dots, P_n$ have distinct non-constant parts.

\begin{lemma}[Degree sequence inductive step]\label{degree sequence induction}
Suppose that $k > 1$.  Then for any $H_1 \leq M$ there exists an interval $I' \subset I$ along with 1-bounded functions $g_0, \dots, g_{n'}$ supported on $[N]$ and polynomials $Q_1, \dots, Q_{n'}$ with  $D(\vQ) \prec D(\vP)$
 such that
\begin{multline}\label{induction conclusion}
 \Bigabs{ \recip{NM}\sum_{x} \sum_{y \in I} f_0(x)f_1\bigbrac{x+P_1(y)} \dotsm f_n\bigbrac{x+P_n(y)}}\ll \\ nH_1^{-1/2} + 
\Bigabs{\recip{NM}\sum_x \sum_{y \in I'} g_0(x) g_1\bigbrac{x+Q_1(y)} \dotsm g_{n'}\bigbrac{x+Q_{n'}(y)} }^{1/2}.
\end{multline}
Moreover, we can ensure that 
\begin{itemize}
\item $g_{n'} = f_n$,
\item $n' \leq 2n$,
\item $0, Q_1, \dots, Q_{n'}$ have distinct non-constant parts,  
\item $Q_1, \dots, Q_{n'}$ have height at most $ H(4H_1)^{k}$,
\item writing $t$ for the smallest degree such that $D_t(\vP) >0$, we have
\begin{equation}\label{colex bound}
D(\vQ) = \Bigbrac{i_1, \dots, i_{t-1}, D_t(\vP)-1, D_{t+1}(\vP), D_{t+2}(\vP), \dots}
\end{equation}
for some  $i_1 + \dots + i_{t-1} \leq 2n $.
\end{itemize}
\end{lemma}

\begin{proof}
At the cost of increasing the height $H$ by a factor of 2, we may assume that the polynomial $P_n$ occurring in the argument of the function $f_n$ has maximal degree $k > 1$.  To see why this is so, suppose that the maximal index $j$ with $\deg P_j = k$ satisfies $j < n$.  Then performing the change of variables $x \mapsto x-P_j(y)$ results in a configuration $\tilde{\vP}$ of the required form satisfying 
$$
\sum_{x} \sum_{y \in I} f_0(x)f_1\bigbrac{x+P_1(y)} \dotsm f_n\bigbrac{x+P_n(y)}\\ = \sum_{x} \sum_{y \in I} f_j(x)f_1\bigbrac{x+\tilde{P}_1(y)} \dotsm f_0\bigbrac{x + \tilde{P}_j(y)}\dotsm f_n\bigbrac{x+\tilde{P}_n(y)}
$$
 Moreover, $\tilde{\vP}$ has height at most $2H$ and $0, \tilde{P}_1, \dots, \tilde{P}_n$ have distinct non-constant parts.  

Given this assumption, 
let us re-arrange the remaining indices with respect to the order of $\deg P_i$, so that there exists an index $l \leq n$ with
\begin{equation}\label{post reordering}
\deg P_i =1 \iff i < l \qquad \text{and} \qquad \deg P_l = \min\set{\deg P_i :i \geq l}.
\end{equation}

\begin{claim}
There are at most $n^2$ choices of $h$ for which the following polynomials have indistinct non-constant parts
\begin{equation}\label{distinct polynomials}
P_1(y), \dots, P_n(y), P_l(y+h), \dots, P_n(y+h)
\end{equation}
\end{claim}
To establish the claim, let us suppose that $h$ is such that two of the polynomials in the list \eqref{distinct polynomials} have the same non-constant part.   Since the polynomials $P_1, \dots, P_n$ have distinct non-constant parts, the only possibility is that $P_i(y+h) - P_j(y)$ is constant for some $l \leq i \leq n$ and $1 \leq j \leq n$.  Let $P_i(y) = a_d y^d + a_{d-1}y^{d-1} + \dots$ with $a_d \neq 0$.  Then since $d > 1$ we have
$$
P_i(y+h) = a_dy^d + (da_d h + a_{d-1}) y^{d-1} + \dots
$$
The expression $d a_d h + a_{d-1}$ must equal  the coefficient of $y^{d-1}$ in $P_j$, and this completely determines $h$.  Since there are at most $n$ choices for $P_j$ and at most $n$ choices for $P_i$ the claim follows.

Let $\mathcal{H}_1$ denote the set of $h$ for which two of the polynomials in \eqref{distinct polynomials} have the same non-constant part.  Notice that $\mathcal{H}_1$ contains $0$.  Applying Lemma \ref{linearisation step}, we deduce that there exists $h \in [H_1] \setminus \mathcal{H}_1$ and an interval $I' \subset I$ such that the left-hand side of \eqref{induction conclusion} is of order at most
\begin{equation}\label{first app of lin}
nH_1^{-1/2} +  \Biggabs{\recip{NM} 
\sum_x \sum_{y \in I'}  \prod_{\substack{1 \leq i \leq n\\ \omega \in \set{0,1}}}f_i\bigbrac{x+P_i(y + \omega h) - P_1(y)}}^{1/2} .
\end{equation}

Using the notation \eqref{difference operator} we see that for each $i < l$ there exists an integer $a_i = P_i(h)$ such that
$$
f_i\bigbrac{x+P_i(y ) - P_1(y)} f_i\bigbrac{x+P_i(y + h) - P_1(y)} = \Delta_{a_i}f_i\bigbrac{x +(P_i-P_1)(y)}.
$$
For such values of $i$ let us write $g_{i-1} := \Delta_{a_i} f_i$ and $Q_{i-1} := P_i - P_1$.  For the remaining indices, set
\begin{itemize}
\item $g_{l+2i-1} = g_{l + 2i}    := f_{l+i}$, 
\item $Q_{l+2i-1}(y) := P_{l+i}(y)- P_1(y)$,
\item $Q_{l+2i}(y) := P_{l+i}(y+h) - P_1(y)$;
\end{itemize}
where in each case $i$ ranges over $0 \leq i \leq n-l$.
Then one can check that
$$
g_0(x) g_1\bigbrac{x + Q_1(y)}\dotsm g_{n'}\bigbrac{x + Q_{n'}(y)} = \\  \prod_{\substack{1 \leq i \leq n\\ \omega \in \set{0,1}}}f_i\bigbrac{x+P_i(y + \omega h) - P_1(y)},
$$ which yields \eqref{induction conclusion} with $n' = 2n - l \leq 2n$.

From our claim we see that $0, Q_1, \dotsm, Q_{n'}$ have distinct non-constant parts, since adding $P_1$ to each polynomial in this sequence gives the sequence \eqref{distinct polynomials}.  Also $g_{n'} = g_{l + 2(n-l)} = f_n =  f_A$.  From a crude estimate using the binomial theorem, one can check that the height of each $Q_i$ is at most $\trecip{2}H(4H_1)^{k}$.

It remains to show that $D(\vQ)$ has the form given in \eqref{colex bound}, and consequently $D(\vQ) \prec D(\vP)$.  From \eqref{post reordering} we have $t =\deg P_1$.  Hence if $\deg P_i > t$ then for either choice of $\omega \in \set{0,1}$, the polynomial $P_i(y+\omega h) - P_1(y)$ has the same leading term as $P_i$.  It follows that for $s > t$ we have $D_s(\vQ) = D_s(\vP)$.  Let $\set{a_1, \dots, a_r}$ denote the set of leading coefficients which appear in some $P_i$ with $\deg P_i = t$.  We may assume that $a_1$ is the leading coefficient of $P_1$.  Then the set of leading coefficients occurring amongst those $Q_i$ with $\deg Q_i  = t$ is equal to $\set{a_2 - a_1, \dots, a_r - a_1}$, which has cardinality one less than $\set{a_1, \dots, a_r}$.  Moreover, since there are at most $2n$ polynomials $Q_i$, the number of $Q_i$ with $\deg Q_i < t$ is also at most $2n$.  This leads to the bound for $i_1 + \dots + i_{t-1}$ claimed in the theorem.
\end{proof}

\begin{lemma}[Full linearisation]\label{full degree sequence induction}
Writing $\vm := D(\vP)$, there exist positive integers $R = R(n, \vm)$, $r 
\leq R$ and $d 
 \leq 2^rn$ such that for any $H_1 \leq M$ there exists an interval $I' \subset I$ along with 1-bounded functions $g_0, \dots, g_{d}$ supported on $[N]$ such that $g_d = f_n$ and
\begin{multline}\label{full induction conclusion}
\Bigabs{\recip{NM} \sum_x \sum_{y \in I} f_0(x) f_1\bigbrac{x+P_1(y)} \dotsm f_n\bigbrac{x+P_n(y)}} \ll nH_1^{-1/2^r} +\\ \Bigabs{\recip{NM}\sum_x \sum_{y \in I'} g_0(x)g_1(x+a_1y+b_1)\dotsm g_d(x+a_d y +b_d) }^{1/2^r},
\end{multline}
for some integers $a_i,b_i$ with $a_i$ distinct, non-zero, and of magnitude at most $H(4H_1)^{rk}$.  
\end{lemma}

\begin{remark}
It is important for our purposes that whilst the numbers $r$ and $d$ may depend on the coefficients of the configuration $\vP$, we have upper bounds for these quantities which depend solely on $D(\vP)$ and $n$.
\end{remark}

\begin{proof}
We proceed by induction along the colex order of $\vm:= D(\vP)$, proving the inequality \eqref{full induction conclusion} with absolute constant $8C^2$, where $C$ is the absolute constant occurring in \eqref{induction conclusion}.

  If $k = \max_i \deg P_i = 1$ then we are done on taking $r(\vm) = 1$ and $d = n$.  Let us therefore assume that $k := \max_i \deg P_i > 1$ and apply Lemma \ref{degree sequence induction} to conclude the existence of:
\begin{itemize}
\item  an interval $I' \subset I$,
\item 1-bounded functions $g_0, \dots, g_{n'}$ supported on $[N]$ with $n' \leq 2n$ and $g_{n'} = f_n$,
\item  polynomials $0, Q_1, \dots, Q_{n'}$ of height at most $H(4H_1)^{k}$ and distinct non-constant parts,
\end{itemize}
such that together these satisfy the inequality
\begin{multline}
 \Bigabs{ \recip{NM}\sum_{x} \sum_{y \in I} f_0(x)f_1\bigbrac{x+P_1(y)} \dotsm f_n\bigbrac{x+P_n(y)}}\leq CnH_1^{-1/2} + \\
C\Bigabs{\recip{NM}\sum_x \sum_{y \in I'} g_0(x) g_1\bigbrac{x+Q_1(y)} \dotsm g_{n'}\bigbrac{x+Q_{n'}(y)} }^{1/2}.
\end{multline}
Furthermore, writing $t$ for the smallest degree such that $m_t > 0$, we have
\begin{equation*}
\vm' := D(\vQ) = (i_1, \dots, i_{t-1}, m_t - 1, m_{t+1}, \dots ) \prec \vm
\end{equation*}
for some $i_1 + \dots + i_{t-1} \leq 2n$.

Applying the induction hypothesis, we conclude that there exist positive integers $R' = R(n', \vm')$, $r' \leq R'$ and $d \leq 2^{r'}n'$ along with 
\begin{itemize}
\item  an interval $I'' \subset I'$;
\item 1-bounded functions $\tilde{g}_0, \dots, \tilde{g}_{d}$ supported on $[N]$ with $d \leq 2^{r'} n'$ and $\tilde{g}_d = g_{n'}=f_n$;
\item  integers $a_i,b_i$ with $a_1, \dots, a_d$ distinct, non-zero and of magnitude at most $H(4H_1)^k(4H_1)^{r'k}$.
\end{itemize}
Moreover, we have the inequality
\begin{multline}\label{induction step ineq}
\Bigabs{\recip{NM}\sum_x \sum_{y \in I'} g_0(x) g_1\bigbrac{x+Q_1(y)} \dotsm g_{n'}\bigbrac{x+Q_{n'}(y)} }\leq 8 C^2 n'H_1^{-1/2^{r'}} +\\ 8 C^2\Bigabs{\recip{NM}\sum_x \sum_{y \in I'} \tilde{g}_0(x)\tilde{g}_1(x+a_1y+b_1)\dotsm \tilde{g}_d(x+a_d y +b_d) }^{1/2^{r'}}
\end{multline}

Setting $r = r'+1$, the observation that $n' \leq 2n$ gives $d \leq 2^r n$ and $|a_i| \leq H(4H)^{rk}$.  Furthermore, \eqref{full induction conclusion} follows with the claimed constant from the inequality
\begin{equation*}
\begin{split}
C n H_1^{-1/2} + C\brac{8C^2n' H_1^{-1/2^{r'}}}^{1/2} &\leq (Cn+C(16 C^2 n)^{1/2})H_1^{-1/2^r}\\
& \leq 8C^2 n H_1^{-1/2^r}.
\end{split}
\end{equation*}

It remains to establish the existence of $R(n, \vm)$.  Define $\mathcal{M}(n, \vm)$ to be the set 
$$
 \set{\vm' : m_1' + \dots + m_{t-1}' \leq 2n ,\ m_t' = m_t -1, \ m_j' =m_j \text{ for } j>t}.
$$
Since $t$ is determined by $\vm$, the set $\mathcal{M}(n, \vm)$ is completely determined by $n$ and $\vm$.  By induction along colex, the integer $R(n', \vm')$ exists for each $\vm' \in \mathcal{M}(n, \vm)$ and any valid choice of $n' \leq 2n$, hence the  lemma follows on defining
\begin{equation*}
\begin{split}
R(n, \vm) := 1 + \max_{\substack{n' \leq 2n\\ \vm' \in \mathcal{M}(n , \vm)}} R(n', \vm').
\end{split}
\end{equation*}

\end{proof}

\begin{corollary}[Linearisation for $k$th power configurations]\label{linearisation induction}
Let $f_0$, $\dots$, $f_n : \Z \to [-1, 1]$ be 1-bounded functions supported on $[N]$ and let $c_1,\dots, c_n$ be distinct non-zero integers.  Then there exist integers $r = r(n, k)$ and $d = d(n,k)$ such that for any $H \leq N^{1/k}$ there exists $M \leq N^{1/k}$ for which 
\begin{multline}\label{8-term conclusion}
\Bigabs{N^{-\frac{k+1}{k}} \sum_{x\in \Z} \sum_{y \in \N} f_0(x) f_1\brac{x+c_1y^k} \dotsm f_n\brac{x+c_ny^k}} \ll nH^{-1/2^r} +\\ \Bigabs{N^{-\frac{k+1}{k}}\sum_x \sum_{y \in [M]} g_0(x)g_1(x+a_1y+b_1)\dotsm g_d(x+a_d y +b_d) }^{1/2^r},
\end{multline}
where $g_0, \dots, g_{d}$ are 1-bounded functions supported on $[N]$ with $g_d = f_n$, and $a_i,b_i$ are integers with $a_i$ distinct, non-zero, and of magnitude at most $O_{\vc}\brac{H^{r}}$.
\end{corollary}

\begin{proof}
The conclusion follows from Lemma \ref{full degree sequence induction} and the observation that for this particular configuration $\vP$ we have
$$
D(\vP) = (0, \dots, 0, n,0, \dots)
$$
where the only non-zero entry occurs in the $k$th place.  

To be more precise, the conclusion follows with a height of $O_{\vc}(H_1^{kr})$ rather than the claimed $O_{\vc}(H_1^{r})$.  However, we may increase $r$ to $kr$ and the conclusion remains valid.
\end{proof}

\section{The localised von Neumann theorem}\label{local von neu section}
In this section we show how a function $g_d$ which has large linear average of the form \eqref{8-term conclusion} also has large $U^d$-norm on many short intervals.  Recall the definition of the Gowers norm (and its localisation) given in \eqref{gowers norm}.

\begin{lemma}[Linear local von Neumann]\label{constrained von Neumann}\label{constrained von Neumann lemma}
Let  $f_0, \dots, f_d : \Z \to [-1,1]$ be functions supported on $[N]$ and let $a_i, b_{i}$ be integers with the $a_i$ distinct and satisfying $0 < |a_i| \leq H$ for all $i$.  Then for any $1 \leq M_1 \leq M$ there exists $M_2$ in the range $HM_1 \leq M_2 \ll HM_1$ 
such that we have the inequality
\begin{equation}\label{dAP assumption}
 \Bigabs{\sum_x \sum_{y \in [M]} f_0(x)f_1(x+a_1y+b_1)\dotsm f_d(x+a_d y +b_d) } \ll_d  NM_1 + 
  H^2M\sum_x \frac{\norm{f_d}_{U^d(x+[M_2])}}{\norm{1}_{U^d(x+[M_2])}}.
\end{equation}
\end{lemma}

We deduce this from a standard result in which the common difference $y$ is not constrained to lie in a short interval.

\begin{lemma}\label{unconstrained von neumann}
Let $g_0, g_1, \dots, g_d : \Z \to [-1,1]$ be functions supported on $[-N, N]$ and let $\va_2, \dots, \va_d \in \Z^2$ be such that $(1, 0), (0,1), \va_2, \dots, \va_d$ are pairwise linearly independent.  Then for $d\geq 2$ we have
\begin{equation*}\label{gen von neu}
 \Bigabs{\sum_{z_0, z_1} g_0(z_0)g_1(z_1)g_2(\va_2\cdot\vz) \dotsm g_d(\va_d\cdot\vz) } \ll_d N^{2- \frac{d+1}{2^d}}\norm{g_d}_{U^d}.
\end{equation*}
\end{lemma}
 
 \begin{proof}
 We proceed by induction on $d$.  For $d=2$ we use the Fourier transform, as defined in \eqref{fourier transform}.  If we write $(a_0, a_1)$ for $\va_2$, then orthogonality and H\"older's inequality, together with the fact that $a_0$ and $a_1$ are both non-zero, gives 
 \begin{equation*}
\begin{split}
\Bigabs{\sum_{\vz} g_0(z_0) g_1(z_1) g_2(\va\cdot \vz)} & = \Bigabs{\int_\T \hat{g}_0(a_0\alpha)\hat{g}_1( a_1\alpha) \overline{\hat{g}_2(\alpha)} \intd\alpha}\\
& \leq \norm{g_0}_{L^2}\norm{g_1}_{U^2} \norm{g_2}_{U^2}\\
& \ll N^{5/4} \norm{g_2}_{U^2}.
\end{split}
\end{equation*}
Here we have used the identity $\norm{g_i}_{U^2(\Z)} = \norm{\hat{g}_i}_{L^4(\T)}$.

For the induction step, when $d > 2$, let us again write $(a_0, a_1)$ for $\va_2$  and
$$
G(\vz) := g_0(z_0)g_1(z_1)g_3(\va_3\cdot\vz) \dotsm g_{d}(\va_{d}\cdot\vz).
$$
Then we have
\begin{align*}
 \Bigabs{\sum_{z_0, z_1} g_0(z_0)g_1(z_1)g_2(\va_2\cdot\vz) \dotsm g_d(\va_d\cdot\vz) } &= \Bigabs{\sum_{\vz} G(\vz) g_2(a_0 z_0 + a_1z_1)}\\
& =  \Bigabs{\int_\T \hat{G}(a_0\alpha, a_1\alpha) \overline{\hat{g}_2(\alpha)} \intd\alpha}\\
& \leq \bignorm{\hat{g}_2}_{L^2} \Biggbrac{\int_\T |\hat{G}(a_0\alpha, a_1\alpha)|^2 \intd\alpha}^{1/2}.
\end{align*}
Interpreting the underlying equations, we see that
\begin{align*}
\int_\T |\hat{G}(a_0\alpha, a_1\alpha)|^2 \intd\alpha = \sum_{\va_2\cdot( \vz - \vw) = 0} G(\vz) G(\vw)
 = \sum_{\vh\ :\ \va_2 \cdot \vh = 0}\ \sum_{\vz} G(\vz) G(\vz + \vh).\label{difference count}
\end{align*}

For fixed $\vh$, let us set $\tilde{g}_0 = \Delta_{h_0}g_0$, $\tilde{g}_1 = \Delta_{h_1}g_1$ and $\tilde{g}_i = \Delta_{\va_i\cdot \vh}\, g_i$ for $i \geq 3$.  Then applying the induction hypothesis we have
\begin{align*}
 \sum_{\vz} G(\vz) G(\vz + \vh)&=   \sum_{\vz} \tilde{g}_0(z_0) \tilde{g}_1(z_1) \tilde{g}_3(\va_3 \cdot \vz) \dotsm \tilde{g}_{d}(\va_{d}\cdot \vz)\\
& \ll_d N^{2-d2^{1-d}}  \norm{\Delta_{\va_d\cdot \vh}\, g_d}_{U^{d-1}}.
\end{align*}

Let $c = \hcf(a_0, a_1)$ and set $\vb = (b_0, b_1) := c^{-1} (a_1, -a_0) \in \Z^2$.  Then we know that the set
$$
\set{\vh \in \Z^2 : \va_2 \cdot \vh = 0}
$$
is in bijective correspondence with $\Z$ via the map $h \mapsto h\vb$.  Since $\va_d$ is not colinear to $\va_2$, we have $\va_d \cdot \vb \neq 0$, so we are legitimate in the assertion that 
\begin{equation*}
\begin{split}
\sum_{\vh\ :\ \va_2 \cdot \vh = 0}\  \norm{\Delta_{\va_d\cdot \vh}\, g_d}_{U^{d-1}} & =
\sum_h \norm{\Delta_{h \va_d\cdot \vb}\, g_d}_{U^{d-1}}\\
& \leq \sum_h \norm{\Delta_{h }\, g_d}_{U^{d-1}}.
\end{split}
\end{equation*}

Given that $g_d$ is supported on $[-N, N]$, the set of integers $h$ for which $\Delta_h\, g_d$ is not identically zero is contained in the interval $[-2N, 2N]$.  Hence by H\"{o}lder's inequality 
\begin{align*}
 \brac{\sum_h \norm{\Delta_h\, g_d}_{U^{d-1}}}^{2^{d-1}}& \leq (4N+1)^{2^{d-1}-1} \sum_h \norm{\Delta_h\, g_d }_{U^{d-1}}^{2^{d-1}}\\
& = (4N+1)^{2^{d-1}-1}  \norm{g_d}_{U^{d}}^{2^{d}}.
\end{align*}
Thus
\begin{align*}
& \Bigabs{\sum_{z_0, z_1} g_0(z_0)g_1(z_1)g_2(\va_2\cdot\vz) \dotsm g_d(\va_d\cdot\vz) }  \ll_d N^{\recip{2} + 1 - d2^{-d} + \frac{1}{2} - 2^{-d}} \norm{g_d}_{U^{d}}.
\end{align*}
 \end{proof}

 \begin{proof}[Proof of Lemma \ref{constrained von Neumann}]  {\ }

 Given $g : \Z \to [-1, 1]$ one can check that for any $1 \leq M_1 \leq M$ we have
 \begin{align*}
 \sum_{y \in [M]} g(y) &= \recip{M_1^2} \sum_{z_0, z_1 \in [M_1]}\ \sum_{y \in [M] + z_0 - z_1} g(y - z_0 + z_1) \\
&  = \recip{M_1^2} \sum_{z_0, z_1 \in [M_1]}\ \sum_{y \in [M] } g(y - z_0 + z_1) + O(M_1).
\end{align*}
Applying this to the left-hand side of \eqref{dAP assumption} and maximising over $y \in [M]$, we deduce that 
\begin{multline*}
 \Bigabs{\sum_x \sum_{y \in [M]} f_0(x)f_1(x+a_1y+b_1)\dotsm f_d(x+a_d y +b_d) } \ll NM_1 + \\
\frac{M}{M_1^2}\Bigabs{ \sum_x\sum_{z_0, z_1 \in [M_1]}  f_0(x)\tilde{f}_1\bigbrac{x+a_1(z_1 - z_0 )}\dotsm \tilde{f}_d\bigbrac{x+a_d ( z_1- z_0) } }.
\end{multline*}
Shifting the $x$ variable by $a_1 z_0$ and setting $\va_i := (a_1 - a_i, a_i)$ gives
\begin{multline*}
  \sum_x\sum_{z_0, z_1 \in [M_1]}  f_0(x)\tilde{f}_1(x+a_1(z_1 - z_0 ))\dotsm \tilde{f}_d(x+a_d ( z_1- z_0) ) = \\
   \sum_x\sum_{z_0, z_1 \in [M_1]}  f_0(x+a_1z_0)\tilde{f}_1(x+a_1z_1)\tilde{f}_2(x + \va_2 \cdot \vz)\dotsm \tilde{f}_d(x+\va_d\cdot \vz)   .
\end{multline*}

For fixed $x$, write
$$
g_{0}(z) := f_0(x + a_1z) 1_{[M_1]}(z), \qquad g_{1}(z) := \tilde{f}_1(x + a_1z ) 1_{[M_1]}(z)
$$
and for $i \geq 2$ set
$$
 g_{i}(z) := \tilde{f}_i(x + z) 1_{I}(z) \quad \text{where} \quad I := [-3HM_1, 3HM_1].
$$
Then our height estimate $|a_i| \leq H$ ensures that
\begin{multline*}
  \sum_{z_0, z_1 \in [M_1]}  f_0(x+a_1z_0)\tilde{f}_1(x+a_1z_1)\tilde{f}_2(x + \va_2 \cdot \vz)\dotsm \tilde{f}_d(x+\va_d\cdot \vz) = \\
  \sum_{z_0, z_1 }  g_0(z_0)g_1(z_1)g_2(\va_2 \cdot \vz)\dotsm g_d(\va_d\cdot \vz)   .
\end{multline*}

By Lemma \ref{unconstrained von neumann}
\begin{align*}
 \sum_{z_0, z_1 }  g_0(z_0)g_1(z_1)g_2(\va_2 \cdot \vz)\dotsm g_d(\va_d\cdot \vz)  \ll_d \brac{HM_1}^{2 - \frac{d+1}{2^d}} \norm{g_{d}}_{U^d}.
\end{align*}
Writing $M_2$ for the number of integers in $I$, the result follows on noting the lower bound  $\norm{1}_{U^d[M_2]}^{2^d} \gg_d M_2^{d+1}$ together with the shift invariance 
$$
\sum_x \bignorm{\tilde{f}_d}_{U^d(x + I)} = \sum_x \norm{f_d}_{U^d(x + [M_2])}. 
$$
 \end{proof}

Combining Lemma \ref{constrained von Neumann} and the linearisation process (Corollary \ref{linearisation induction}), we obtain the required generalised von Neumann theorem.

 \begin{corollary}[Local von Neumann theorem]\label{local von neumann}
 Let $f_0$, $\dots$, $f_n : \Z \to [-1, 1]$ be 1-bounded functions supported on $[N]$ and let $c_1,\dots, c_n$ be distinct non-zero integers.  Then there exist integers $r = r(n, k)$ and $d = d(n,k)$ such that for any $H\leq N^{1/k}$ and any $0 \leq i \leq n$ there exists $M \ll_{\vc} H^r N^{1/k}$ satisfying
$$
\Bigabs{N^{-\frac{k+1}{k}} \sum_{x\in \Z} \sum_{y \in \N} f_0(x) f_1\brac{x+c_1y^k} \dotsm f_n\brac{x+c_ny^k}}\ll_{\vc, k}  H^{-1/2^r} +\Biggbrac{\frac{H^{r+1}M}{N^{\frac{k+1}{k}}}\sum_x \frac{\norm{f_i}_{U^d(x+[M])}
}{\norm{1}_{U^d(x+[M])}
}}^{1/2^r}.
$$
\end{corollary}
 
 \begin{proof}
 
 Without loss of generality, we may assume that $i = n$.  This is clear on re-ordering indices if $i >0$.  If $i = 0$ then we perform the change of variables 
\begin{multline*}
 \sum_{x\in \Z} \sum_{y \in \N} f_0(x) f_1\brac{x+c_1y^k} \dotsm f_n\brac{x+c_ny^k} = \\ \sum_{x\in \Z} \sum_{y \in \N} f_0\brac{x - c_n y^k} f_1\brac{x+(c_1-c_n)y^k} \dotsm f_{n-1} \brac{x + (c_{n-1} - c_n) y^k}f_n\brac{x},
\end{multline*}
noting that the integers $-c_n, c_1 - c_n, \dots, c_{n-1} - c_n$ are distinct and non-zero.
 
By Corollary \ref{linearisation induction} there exist integers $r = r(n, k)$ and $d = d(n,k)$ such that for any $H \leq N^{1/k}$ there exists $M \leq N^{1/k}$ for which 
\begin{multline*}
\Bigabs{N^{-\frac{k+1}{k}} \sum_{x\in \Z} \sum_{y \in \N} f_0(x) f_1\brac{x+c_1y^k} \dotsm f_n\brac{x+c_ny^k}} \ll_n H^{-1/2^r} +\\ \Bigabs{N^{-\frac{k+1}{k}}\sum_x \sum_{y \in [M]} g_0(x)g_1(x+a_1y+b_1)\dotsm g_d(x+a_d y +b_d) }^{1/2^r},
\end{multline*}
where $g_0, \dots, g_{d}$ are 1-bounded functions supported on $[N]$ with $g_d = f_n$, and $a_i,b_i$ are integers with $a_i$ distinct, non-zero, and of magnitude at most $O_{\vc}\brac{H^{r}}$.
 
Set $M_1 := \min\set{ \floor{N^{1/k}/H}, M}$.  Then by Lemma \ref{constrained von Neumann} there exists $H^r M_1 \leq M_2 \ll_{\vc} H^rM_1$ for which 
$$
 \Bigabs{\sum_x \sum_{y \in [M]} g_0(x)g_1(x+a_1y+b_1)\dotsm g_d(x+a_d y +b_d) } \ll_{d, \vc}  NM_1 + 
  H^{2r}M\sum_x \frac{\norm{f_d}_{U^d(x+[M_2])}}{\norm{1}_{U^d(x+[M_2])}}.
$$
The result follows on combining this with the estimate
$$
M_2 \geq H^rM_1 \gg H^r \min\set{N^{1/k}/H, M} \geq H^{r-1} M.
$$
\end{proof}
 
\section{Modifying Gowers's local inverse theorem}\label{mod gowers section}

In the course of re-proving Szemer\'edi's theorem, Gowers \cite[Theorem 18.1]{gowers01} established the following local inverse theorem for the uniformity norm $\norm{\cdot}_{U^d}$.

\begin{gowersthm}For $d \geq 1$ there exist constants $C = C(d)$  and $c = c(d) >0$ such that the following is true.  Suppose that
$f : \Z \to [-1,1]$ satisfies 
\begin{equation*}\label{ModUdBound:eqn}
\norm{f}_{U^d[N]} \geq \delta \norm{1}_{U^d[N]}.
\end{equation*}  
Then one can partition $[N]$ into arithmetic progressions $P_i$, of average length at least $c\delta^CN^{c\delta^{C}}$ such that
\begin{equation}\label{gowers conclusion}
\begin{split}
& \sum_i \norm{f}_{U^1(P_i)} \geq c \delta^{C} \sum_i \norm{1}_{U^1(P_i)}.
\end{split}
\end{equation}

\end{gowersthm}

This result implies that a set $A \subset [N]$ lacking $d+1$ elements in arithmetic progression has size bound $|A| \ll N(\log \log N)^{-\kappa}$ for some small positive constant $\kappa = \kappa(d)$.  The precise value of $\kappa(d)$ depends very much on the permissible value of $C = C(d)$ in the local inverse theorem, which Gowers explicitly calculates.  For polynomial progressions of the form \eqref{main config}, the number of iterative steps required in the linearisation process of \S \ref{linearisation section} means that the degree of $d = d(n,k)$ of the Gowers norm controlling this configuration grows inordinately rapidly in $n$ and $k$; so much so that we have refrained from estimating it, rendering explicit estimates of $C(d)$ tangential to our purpose.  As a consequence, our final density bound \eqref{loglog bound} has an inexplicit exponent of $\log\log N$.

Unfortunately we cannot use Gowers's inverse theorem as stated.  Our difficulty is that the theorem gives us information about the $U^1$-norm of a function localised to arithmetic progressions, yet we require these progressions to take a special form.

\begin{definition}[$k$th power progression]
Call an arithmetic progression a \emph{$k$th power progression} if it has common difference equal to a perfect $k$th power.
\end{definition}

We require the following modified version of Gowers's inverse theorem.

\begin{theorem}[Gowers's inverse theorem for $k$th power progressions]\label{ModInverse:thm}
For $d, k \geq 1$ there exist $C = C(d,k)$ and $c = c(d,k) >0$ such that the following is true.  Suppose that
$f : \Z \to [-1,1]$ satisfies 
\begin{equation*}
\norm{f}_{U^d[N]} \geq \delta \norm{1}_{U^d[N]}.
\end{equation*}  
Then one can partition $[N]$ into $k$th power progressions $P_i$, of average length at least $c\delta^CN^{\exp(-1/c\delta^{C})}$ such that
\begin{align*}
& \sum_i \norm{f}_{U^1(P_i)} \geq c \delta^{C} \sum_i \norm{1}_{U^1(P_i)}.
\end{align*}
\end{theorem}

The proof of Theorem \ref{ModInverse:thm} follows from an elementary, albeit lengthy, modification of the argument of Gowers \cite[pp.489-585]{gowers01}.  The reader is referred to \cite[\S 5]{green02} for a detailed exposition of the argument for the $U^3$-norm with $k =2$.  Below we sketch the content of the main modification needed in the general case.

Gowers's argument begins by working over the progression $[N]$ (identified with a subset of the group $\Z/N'\Z$ for some prime $N' \ll_d N$) and proceeds by repeatedly passing to (integer) subprogressions, finally obtaining the subprogressions $P_i$ of the conclusion \eqref{gowers conclusion}. This subprogression refinement takes place in 
\cite[\S 16]{gowers01} 
(with an additional refinement taking place in 
\cite[Prop.~17.7]{gowers01}),
repeatedly employing results from \cite[\S5 \& \S 7]{gowers01}  to obtain the required subprogression.  Each stage passes from a progression of common difference $m$ to a progression of common difference $mq$, where $q$ arises in one of the following three ways.
\begin{enumerate}[({A}1)]
\item  Passage to a shorter segment of the same progression, as in
\cite[Prop.~17.7]{gowers01},
so that $q = 1$.
\item\label{dioph approx}  An application of results in
\cite[\S5]{gowers01},
all of which ultimately rest on the following consequence of Weyl's inequality:  There exists $c_{d} > 0$ such that for any $\alpha \in \T$ and $Q \geq 1$ we have
\begin{equation}\label{dioph approx eqn}
\min_{1 \leq q \leq Q}\norm{\alpha q^d} \ll_d Q^{-c_d}.
\end{equation}
See for example \cite[Lem.~16.1]{gowers01}.
\item\label{bohr set difference}  An application of the fact that the Bohr set 
\begin{equation}\label{bohr set}
B(K, \eta ) := \set{ x \in [-N/2, N/2): \norm{\alpha x} \leq \eta \quad (\alpha \in K)}
\end{equation}
contains an arithmetic progression of length $\gg \eta N^{1/(1+|K|)}$.  See \cite[Cor.~7.9--7.10, Lem.~13.4, etc]{gowers01}. 
\end{enumerate}
If $m$ and $q$ are both perfect $k$th powers then it follows that $mq$ is a perfect $k$th power.  Since this iteration begins with $m = 1$, which is itself a perfect $k$th power, it suffices to verify that one can take $q$ equal to a perfect $k$th power in (A\ref{dioph approx}) and (A\ref{bohr set difference}). 

This easily follows for (A\ref{dioph approx}) by replacing $d$ with $dk$ in \eqref{dioph approx eqn}, so that
$$
 \min_{1 \leq q \leq Q} \norm{\alpha q^{dk}} \ll_{d, k} Q^{-c_{d, k}}.
$$
Quantitatively, this replaces the absolute constants $c(d), C(d)$ in Gowers's inverse theorem with constants $c(d,k), C(d, k)$ .

We replace (A\ref{bohr set difference}) with the following.
\begin{lemma}\label{bohr set kth prog}
There exists an absolute constant $C = C(k)$ such that the Bohr set $B(K, \eta)$ defined in \eqref{bohr set} contains a $k$th power progression of length at least
\begin{equation*}
\begin{split}
\gg_k \eta N^{\exp(-C|K|)}.
\end{split}
\end{equation*}
\end{lemma}
In the proof of Gowers's inverse theorem appearing in \cite{gowers01}, all applications of (A\ref{bohr set difference}) have $\eta^{-1} \leq C \delta^{-C}$ and $|K| \leq C\delta^{-C}$ for some absolute constant $C = C(d)$.  This leads to the passage from an arithmetic progression of length $N$ to a progression of length $ c \delta^CN^{c\delta^{C}}$.  Replacing (A\ref{bohr set difference}) with Lemma \ref{bohr set kth prog} therefore passes to an arithmetic progression of length at least $$c \delta^C N^{\exp(-1/c\delta^{C} )}.$$
This leads to the final lower bound on the length of $k$th power progressions obtained in Theorem \ref{ModInverse:thm}.

Lemma \ref{bohr set kth prog} follows from an application of the following result of Cook \cite{cook}.

\begin{lemma}[Simultaneous $k$th power recurrence]\label{kth recurrence}
There exists an absolute constant $C = C(k)$ such that for any $\alpha_1, \dots, \alpha_r \in \T$ and $Q\geq 1$ we have
\begin{equation}\label{quantitative cook}
\begin{split}
\min_{1 \leq q \leq Q} \max_{1\leq i \leq r} \norm{\alpha_i q^k} \ll_k Q^{-\exp(-Cr)}.
\end{split}
\end{equation}
\end{lemma}

Although \cite[Theorem 1]{cook} has superior dependence on $r$ and $k$ in the exponent of $Q$, it has an implicit constant in \eqref{quantitative cook} depending $r$.  The nature of this dependence is important for our application.  We therefore offer the following elementary proof, with implicit constant independent of $r$, following the arguments of \cite{gowers01, green02}.

\begin{proof}
The $k = 1$ case follows from Kronecker's theorem on simultaneous Diophantine approximation.  Let us therefore suppose that $k \geq 2$.  

By a weak version of \cite[Theorem 1.7]{wooley} (and \cite{heilbronn} for $k = 2$), there is a constant $C_k$ such that for any $\alpha \in \T$ and $Q \geq 1$
\begin{equation}\label{wooley}
\begin{split}
\min_{1 \leq q \leq Q} \norm{\alpha q^k} \leq C_k Q^{-1/k^3}.
\end{split}
\end{equation}
Iteratively applying \eqref{wooley}, for each $1 \leq i \leq r$ we can find
\begin{equation*}
\begin{split}
1\leq q_i \leq Q^{\frac{k^4}{(k^4+1)^i}} \quad \text{with} \quad \norm{\alpha_i q_1^k \dotsm q_i^k} \leq C_kQ^{-\frac{k}{(k^4+1)^i}}.
\end{split}
\end{equation*}
Setting $q := q_1 \dotsm q_r$ we have
\begin{equation*}
\begin{split}
q \leq Q^{1 - (k^4 +1)^{-r}} \leq Q
\end{split}
\end{equation*}
and for each $i$
\begin{equation*}
\begin{split}
\norm{\alpha_i q^k} & \leq \norm{\alpha_i q_1^k \dotsm q_i^k} q_{i+1}^k \dotsm q_{r}^k \\
& \leq C_k Q^{k\brac{-\frac{1}{(k^4+1)^i}+ \frac{k^4}{(k^4+1)^{i+1}} + \dots + \frac{k^4}{(k^4+1)^{r}}}}\\
& = C_k Q^{- k(k^4+1)^{-r}}\\
& \leq C_k Q^{- k^{-5r}}.
\end{split}
\end{equation*}
The last estimate following from the inequality $k^4 +1 \leq k^5$ when $k \geq 2$.  The lemma follows with $C = 5\log k$.
\end{proof}

\begin{proof}[Proof of Lemma \ref{bohr set kth prog}]
Let $r := |K|$. Employing Lemma \ref{kth recurrence}, there exists $1 \leq q \leq Q$ such that for any $\alpha \in K$ we have
$$
\norm{\alpha q^k } \leq C Q^{-\exp(-Cr)}.
$$ 

Let $L$ be the largest non-negative integer satisfying $L Q < N/2$ and $LC Q^{-\exp(-Cr)} < \eta$.  Then $B(K, \eta)$ (as defined in \eqref{bohr set}) contains the progression $$\set{ -Lq^k, \dots, -q^k, 0, q^k, \dots, Lq^k},$$ which has length 

 \begin{equation*}
\begin{split}
2L + 1 & \geq \min\set{N/(2Q), (\eta/C) Q^{\exp(-Cr)}}\\
& \gg_k \min\set{ N/Q, \eta Q^{\exp(-Cr)}}.
\end{split}
\end{equation*}

Taking 
$$
Q := (N/\eta)^{\recip{1 + \exp(-Cr)}},
$$ so that $N/Q = \eta Q^{\exp(-Cr)}$, we deduce that 
$
2L+1 \gg_k \eta N^{\exp(-Cr)} 
$
(after increasing $C$ a little).
\end{proof}

\section{The density increment and  final iteration}\label{iteration}

In this section we combine all of our previous work to establish the following density increment lemma.

\begin{lemma}[Density increment lemma]\label{density increment lemma} There exist absolute constants $C = C(n, k)$ and $c = c(\vc, k) > 0$ such that the following is true.  Let $A$ be a subset of $[N]$ of size $\delta N$ which lacks configurations of the form 
\begin{equation}\label{kth power configuration}
x,\ x+c_1y^k,\ \dots,  x+c_ny^k \quad \text{with} \quad y \in \Z\setminus\set{0}.
\end{equation}
Suppose that
\begin{align}\label{exp density increment N bound} 
& N \geq \exp\exp(1/(c\delta^{C})).
\end{align} 
Then there exists a $k$th power progression $P$ such that
\begin{align*}
|P| \geq N^{\exp(-1/c\delta^{C})} \qquad \text{and} \qquad |A\cap P| \geq (\delta +c\delta^{C}) |P|.
\end{align*}
\end{lemma}

\begin{proof}  
Throughout the following proof we write $c = c(\vc, k) >0$ and $C = C(n,k)$ for absolute constants dependent only on their respective parameters.  Different occurrences of $c$ and $C$ may denote different absolute constants.

Since $A$ lacks configurations of the form \eqref{kth power configuration}, we have
\begin{align*}
& \sum_{x \in \Z} \sum_{y\in \N} 1_A(x) 1_A(x+c_1y^k)\dotsm 1_A(x+c_ny^k) = 0.
\end{align*}
Recall the definition of the balanced function 
\begin{align*}
& f_A := 1_A - \delta 1_{[N]}.
\end{align*}
Making the substitution $1_A = \delta 1_{[N]} + f_A$ and expanding, we deduce that there exist $f_0, f_1, \dots, f_n$ satisfying
$$
\set{f_A} \subset \set{f_0,f_1, \dots, f_n} \subset \set{f_A, \delta 1_{[N]}}
$$ 
and 
\begin{multline}\label{eqn:7factor}
 (2^{n+1}-1)  \abs{\sum_{x \in \Z} \sum_{y \in \N} f_0(x) f_1(x+c_1y^k) \dotsm f_n(x+c_ny^k)} \geq \\
  \delta^{n+1} \sum_{x \in \Z} \sum_{y \in \N}1_{[N]}(x) 1_{[N]}(x+c_1y^k)\dotsm 1_{[N]}(x+c_ny^k).
\end{multline}

Notice that \eqref{exp density increment N bound} implies the much weaker estimate
\begin{equation}\label{density increment N bound}
\begin{split}
N \geq 1/(c\delta^C).
\end{split}
\end{equation}
Taking $c = c(\vc, k)$ in \eqref{density increment N bound} sufficiently small, our assumption ensures that
\begin{align*}
& \sum_{x \in \Z} \sum_{y \in \N}1_{[N]}(x) 1_{[N]}(x+c_1y^k)\dotsm 1_{[N]}(x+c_ny^k) \geq c N^{1+ \recip{k}}
\end{align*}
Incorporating this into \eqref{eqn:7factor}, we deduce that 
\begin{align}\label{eqn:mod7factor}
& \abs{\sum_{x \in \Z} \sum_{y \in \N} f_0(x) f_1(x+c_1y^k) \dotsm f_n(x+c_ny^k)} \geq c \delta^{n+1} N^{1+ \recip{k}}.
\end{align}

Applying the local von Neumann theorem (Corollary \ref{local von neumann}), there exist integers $r = r(n, k)$ and $d = d(n,k)$ such that for any $H\leq N^{1/k}$ there is an integer $M \ll_{\vc} H^r N^{1/k}$ satisfying
\begin{equation*}
H^{-1/2^r} +\Biggbrac{\frac{H^{r+1}M}{N^{1+\recip{k}}}\sum_x \frac{\norm{f_A}_{U^d(x+[M])}}{\norm{1}_{U^d(x+[M])}}}^{1/2^r} \geq c \delta^{n+1}.
\end{equation*}
Taking $ H = C(\vc, k) \delta^{-C(n,k)}$ sufficiently large allows us to conclude that
\begin{equation}\label{H factor}
\Biggbrac{\frac{H^{r+1}M}{N^{1+\recip{k}}}\sum_x \frac{\norm{f_A}_{U^d(x+[M])}}{\norm{1}_{U^d(x+[M])}}}^{1/2^r} \geq c \delta^{n+1}.
\end{equation}
In order for such a choice of $H$ to be permissible, we require that 
$$
N^{1/k} \geq C(\vc, k) \delta^{-C(n,k)},
$$ 
which certainly follows from \eqref{density increment N bound} on taking $C$ and $c$ therein sufficiently large and small (respectively).  This choice of $H$ incorporated into \eqref{H factor} then gives
\begin{equation*}
\frac{M}{N^{1/k}}\sum_{x} \frac{\norm{f_A}_{U^d(x+[M])}}{\norm{1}_{U^d(x+[M])}} \geq c\delta^{C} N.
\end{equation*}

We claim that $M$ lies in the range
\begin{equation}\label{M bound}
c\delta^C N^{1/k} \leq M \ll_{\vc} \delta^{-C} N^{1/k}.
\end{equation}
The upper bound follows from the conclusion $M \ll_{\vc} H^r N^{1/k}$ given by Corollary \ref{local von neumann}.  For the lower bound, we first note that the function $f_A 1_{x + [M]}$ is identically zero unless $x \in [N] -  [M]$.  Since $M \ll_{\vc} H^r N^{1/k}$ we may ensure that $M \leq N$ from \eqref{density increment N bound} and the fact that $k \geq 2$ (if $k = 1$ the density increment lemma is proved in \cite{gowers01}). Hence 
\begin{equation}\label{combined bound}
\frac{M}{N^{1/k}}\sum_{|x| < N} \frac{\norm{f_A}_{U^d(x+[M])}}{\norm{1}_{U^d(x+[M])}} \geq c\delta^{C} N.
\end{equation}
The trivial estimate $\norm{f_A}_{U^d(x + [M])} \leq \norm{1}_{U^d(x + [M])}$ then yields the lower bound in \eqref{M bound}.

It also follows from \eqref{M bound} and \eqref{combined bound} that there exists a set $X \subset (-N, N)$ of size $|X| \geq  c\delta^{C} N$ such that each $x \in X$ satisfies
\begin{align*}
 \norm{f_A}_{U^d(x+[M])} \geq c \delta^C \norm{1}_{U^d(x+[M])}.
\end{align*}
Applying Theorem \ref{ModInverse:thm}, we see that for each $x \in X$, there exists a partition of $x + [M]$ into $k$th power arithmetic progressions $P_{x, i}$ $(i \in I(x))$ of average length at least $c\delta^CM^{\exp(-1/c\delta^{C})}$ and such that
\begin{align*}
& \sum_{i \in I(x)} \norm{f_A}_{U^1(P_{x,i})} \geq c \delta^{C} \sum_i \norm{1}_{U^1(P_{x,i})} = c \delta^C M.
\end{align*}
Taking the trivial partition $P_{x,1} := x +[M]$ for $x \notin X$, we conclude that 
\begin{equation}\label{abs correlation}\begin{split}
\sum_{ |x| < N} \sum_{i \in I(x)}\Biggabs{\sum_{y \in P_{x,i}} f_A(y)} \geq c \delta^C NM.
\end{split}\end{equation}

Since $f_A = 1_A - \delta 1_{[N]}$ has mean zero, 
\begin{align*}
\sum_{ |x| < N} \sum_{i \in I(x)}\sum_{y \in P_{x,i}} f_A(y)  &  = \sum_{ x} \sum_{y \in x + [M]} f_A(y) \\
& = \sum_{z \in [M]} \sum_{ x} f_A(x+ z) \\
& = 0.
\end{align*}
Adding the above to \eqref{abs correlation} we find that
\begin{equation}\label{average behaviour}
\sum_{ |x| < N} \sum_{i \in I(x)} \max\Bigset{ \sum_{y \in P_{x,i}} f_A(y), 0 } \geq c \delta^{C} \sum_{ |x| < N}\ \sum_{i \in I(x)} |P_{x,i}|.
\end{equation}

Let $N_x$ denote the average size of the $P_{x,i}$ and let $J(x)$ denote the set of $i \in I(x)$ for which 
$$
|P_{x,i}|\geq \trecip{2}c \delta^{C} N_x.
$$
Since $f_A$ is 1-bounded we can replace $I(x)$ with $J(x)$ in the left-hand side of \eqref{average behaviour}, at the cost of halving the right-hand side.  The pigeon-hole principle then gives the existence of $x$ and $i \in J(x)$ such that 
$$
\max\Bigset{ \sum_{y \in P_{x,i}} f_A(y), 0 } \geq c \delta^{C} |P_{x,i}|.
$$

The lower bound in \eqref{M bound} combines with our lower bound on $N_x$  to give $|P_{x,i}| \geq c\delta^CN^{\exp(-1/c\delta^{C})}$.  This almost gives us our lemma, all under the weak assumption \eqref{density increment N bound}.  The stronger assumption \eqref{exp density increment N bound} ensures that for some $c' \gg c$ we have
$$
c\delta^CN^{\exp(-1/c\delta^{C})} \geq N^{\exp(-1/c'\delta^{C})}.
$$
\end{proof}

The proof of our main theorem quickly follows.

\begin{proof}[Proof of Theorem \ref{main:thm}]
Suppose that $A \subset [N]$ with $|A| = \delta N$ lacks a configuration of the form 
\begin{equation}\label{sqAP config}
x,\ x+c_1y^k,\ \dots, \ x+c_ny^k \quad \text{with}\quad y \in \Z\setminus\set{0}.
\end{equation}
Then by Lemma \ref{density increment lemma}, provided that
$$
N \geq \exp\exp(1/c\delta^C),
$$ 
there exists a $k$th power progression $P = a + q^k \cdot [N_1]$ of length at least $$N^{\exp(-1/c\delta^{C})}$$ such that
\begin{align*}
& |A\cap P| \geq (\delta +c\delta^{C}) |P|.
\end{align*}
Let $A_1 := \set{x \in \Z : a + q^k x \in A \cap P}$.  Then we have obtained a set $A_1 \subset [N_1]$ lacking configurations of the form \eqref{sqAP config} and of density $\delta_1 := |A_1|/N_1$ satisfying $\delta_1 \geq \delta + c \delta^C$.

Setting
$$
\delta_0 := \delta, \qquad N_0 := N, \qquad A_0 := A,
$$
let us iteratively apply Lemma \ref{density increment lemma}.  Provided that
\begin{equation}\label{iterative N bound}
N_i \geq  \exp\exp(1/c \delta_i^C) \qquad (0 \leq i < j),
\end{equation}
there exists a set $A_j \subset [N_j]$ lacking configurations of the form \eqref{sqAP config} and of density $\delta_j := |A_j|/N_j$ satisfying 
\begin{equation}\label{new delta bound}
\delta_{j} \geq \delta_{j-1} + c\delta_{j-1}^C.
\end{equation}
Moreover, we have the length lower bound
\begin{equation}\label{new N bound}
N_{j} \geq N_{j-1}^{\exp(-1/c\delta_{j-1}^{C})}.
\end{equation}

Using \eqref{new delta bound} gives $\delta_j \geq \delta + jc\delta^C$.  Hence if $j \geq 1/(c\delta^C)$  we obtain the contradiction  $\delta_j > 1$.  It follows that \eqref{iterative N bound} cannot hold for $j \geq 1/(c\delta^C)$, so there exists $i \leq 1/(c\delta^C)$ such that
\begin{equation}\label{N_i bound}
N_i <  \exp\exp(1/c \delta_i^C) \leq \exp\exp(1/c \delta^C).
\end{equation} 
By \eqref{new N bound}, we have
\begin{align*}
& N_i \geq  N^{\exp(-i/c\delta^{C})} \geq N^{\exp(-1/(c\delta^{C})^2)}.
\end{align*} 
Altering our values of $c$ and $C$ appropriately, we deduce that
\begin{align*}
& \exp\exp(1/c \delta^C) \geq N^{\exp(-1/c\delta^{C})}.
\end{align*}
Taking logarithms twice then gives
\begin{align*}
&2/(c \delta^C) \geq \log \log N.
\end{align*}
\end{proof}


\section*{Acknowledgments} 
The author would like to thank Ben Green for useful conversations on Gowers's local inverse theorem, and Adam Harper and Julia Wolf for their comments on an earlier draft.



\begin{dajauthors}
\begin{authorinfo}[sp]
  Sean Prendiville\\
  School of Mathematics\\ University of Manchester\\ Manchester
\\ UK\\
  sean.prendiville\imageat{}manchester\imagedot{}ac\imagedot{}uk \\
  \url{http://personalpages.manchester.ac.uk/staff/sean.prendiville/}
\end{authorinfo}
\end{dajauthors}

\end{document}